\documentclass[12pt]{amsart}

\usepackage{amsmath, amsthm, amssymb}
\usepackage{mathtools,  color}

\newcommand{\R}{{\mathbb R}}

\newcommand{\oR}{{\mathbb{R}^{n}}}
\newcommand{\thess}{\mathring{\nabla}^2}

\DeclareSymbolFont{txlargeoperators}{OMX}{txex}{m}{n}
\DeclareSymbolFont{txlargeoperatorsA}{U}{txexa}{m}{n}
\DeclareMathSymbol{\intop}{\mathop}{txlargeoperators}{"52}
\DeclareMathSymbol{\iintop}{\mathop}{txlargeoperatorsA}{33}
\DeclareMathSymbol{\iiintop}{\mathop}{txlargeoperatorsA}{35}
\DeclareMathSymbol{\iiiintop}{\mathop}{txlargeoperatorsA}{37}
\DeclareMathSymbol{\idotsintop}{\mathop}{txlargeoperatorsA}{39}

\DeclareMathSymbol{\ointop}{\mathop}{txlargeoperators}{"48}
\DeclareMathSymbol{\oiintop}{\mathop}{txlargeoperatorsA}{8}
\DeclareMathSymbol{\oiiintop}{\mathop}{txlargeoperatorsA}{41}

\DeclareMathSymbol{\varointctrclockwiseop}{\mathop}{txlargeoperatorsA}{43}
\DeclareMathSymbol{\varointclockwiseop}{\mathop}{txlargeoperatorsA}{45}
\DeclareMathSymbol{\oiintctrclockwiseop}{\mathop}{txlargeoperatorsA}{64}
\DeclareMathSymbol{\varoiintclockwiseop}{\mathop}{txlargeoperatorsA}{66}
\DeclareMathSymbol{\oiintclockwiseop}{\mathop}{txlargeoperatorsA}{72}
\DeclareMathSymbol{\varoiintctrclockwiseop}{\mathop}{txlargeoperatorsA}{74}
\DeclareMathSymbol{\oiiintctrclockwiseop}{\mathop}{txlargeoperatorsA}{68}
\DeclareMathSymbol{\varoiiintclockwiseop}{\mathop}{txlargeoperatorsA}{70}
\DeclareMathSymbol{\oiiintclockwiseop}{\mathop}{txlargeoperatorsA}{76}
\DeclareMathSymbol{\varoiiintctrclockwiseop}{\mathop}{txlargeoperatorsA}{78}
\DeclareMathSymbol{\fintop}{\mathop}{txlargeoperatorsA}{62}
\DeclareMathSymbol{\sqiintop}{\mathop}{txlargeoperatorsA}{80}
\DeclareMathSymbol{\sqiiintop}{\mathop}{txlargeoperatorsA}{82}

\makeatletter
\renewcommand{\int}{\DOTSI\intop\ilimits@}
\renewcommand{\iint}{\DOTSI\iintop\ilimits@}
\renewcommand{\iiint}{\DOTSI\iiintop\ilimits@}
\renewcommand{\iiiint}{\DOTSI\iiiintop\ilimits@}
\renewcommand{\idotsint}{\DOTSI\idotsintop\ilimits@}

\renewcommand{\oint}{\DOTSI\ointop\ilimits@}
\newcommand{\oiint}{\DOTSI\oiintop\ilimits@}
\newcommand{\oiiint}{\DOTSI\oiiintop\ilimits@}

\newcommand{\varointctrclockwise}{\DOTSI\varointctrclockwiseop\ilimits@}
\newcommand{\varointclockwise}{\DOTSI\varointclockwiseop\ilimits@}
\newcommand{\oiintctrclockwise}{\DOTSI\oiintctrclockwiseop\ilimits@}
\newcommand{\varoiintclockwise}{\DOTSI\varoiintclockwiseop\ilimits@}
\newcommand{\oiintclockwise}{\DOTSI\oiintclockwiseop\ilimits@}
\newcommand{\varoiintctrclockwise}{\DOTSI\varoiintctrclockwiseop\ilimits@}
\newcommand{\oiiintctrclockwise}{\DOTSI\oiiintctrclockwiseop\ilimits@}
\newcommand{\varoiiintclockwise}{\DOTSI\varoiiintclockwiseop\ilimits@}
\newcommand{\oiiintclockwise}{\DOTSI\oiiintclockwiseop\ilimits@}
\newcommand{\varoiiintctrclockwise}{\DOTSI\varoiiintctrclockwiseop\ilimits@}
\newcommand{\fint}{\DOTSI\fintop\ilimits@}
\newcommand{\sqiint}{\DOTSI\sqiintop\ilimits@}
\newcommand{\sqiiint}{\DOTSI\sqiiintop\ilimits@}
\makeatother

\theoremstyle{plain}
\newtheorem{theorem}{Theorem}[section]
\newtheorem{remark}[theorem]{Remark}

\newtheorem{corollary}[theorem]{Corollary}
\newtheorem{claim}[theorem]{Claim}

\newtheorem{lemma}[theorem]{Lemma}
\newtheorem{proposition}[theorem]{Proposition}

\numberwithin{equation}{section}

\allowdisplaybreaks

\title[Rigidity of Poincar\'e-Einstein manifolds]{Rigidity of Poincar\'e-Einstein manifolds with flat Euclidean conformal infinity}

\author{Sanghoon Lee}
\address{Korea Institute for Advanced Study, Hoegiro 85, Seoul 02455, Korea}

\email{sl29@kias.re.kr}
\author{Fang Wang}
\address{Shanghai Jiao Tong University, 800 Dongchuan Rd, Shanghai 200240, China}
\email{fangwang1984@sjtu.edu.cn}

\begin{document}

\date{\today}

 \keywords{}


\maketitle

\begin{abstract}

In this paper, we prove a rigidity theorem for Poincar\'e-Einstein manifolds whose conformal infinity is a flat Euclidean space. The proof relies on analyzing the propagation of curvature tensors over the level sets of an adapted boundary defining function. Additionally, we provide examples of Poincar\'e-Einstein manifolds with non-compact conformal infinities. Furthermore, we draw analogies with Ricci-flat manifolds exhibiting Euclidean volume growth, particularly when the compactified metric has non-negative scalar curvature.

\end{abstract}

\section{Introduction}

\subsection{Background and main theorems}
Suppose $\overline{X}^{n+1}$ is an $n+1$-dimensional smooth manifold with boundary $\partial X=M$ and interior $X$. 
Let $x\in C^{\infty}(\overline{X})$ be a smooth \textit{boundary defining function}, i.e.
$$
x>0\ \textrm{in $X$}, 
\quad x=0 \ \textrm{on $M$}, 
\quad dx \neq 0 \ \textrm{on $M$}. 
$$
Let $g_+$ be a complete Riemannian metric on $X$.
We say that $g_+$  is a \textit{Poincar\'{e}-Einstein metric} (PE) of  $C^{k,\alpha}$(or $C^{\infty}$) regularity   if it satisfies the Einstein equation
\begin{equation}
	Ric_{g_+}=-ng_+, \quad \textrm{in $X$}
\end{equation}
and if the conformally compactified metric $\bar{g}=x^2g_+$  extends to a $C^{k,\alpha}$(or $C^{\infty}$) metric on $\overline{X}$. Defining $h=\bar{g}|_{M}$,the pair $(M, [h ])$ is called the conformal infinity of $(X, M, g_+). $
Typically, we assume $k\geq 2, 0\leq \alpha<1$, which ensures that the sectional curvature satisfies
$$
K\rightarrow -|dx|^2_{\bar{g}}=-1, \ \textrm{as $x\rightarrow 0$}. 
$$

When  $\overline{X}$ and $M$ are both compact, this class of metrics has been extensively studied in many classical works. Examples include boundary regularity theory \cite{CDLS, An1}, existence theory \cite{An2, GL, Le2, BL, GS}, non-existence theory \cite{GH, GHS}, uniqueness theory \cite{ST, DJ, LQS, CLW1, CGJQ}, and compactness theory \cite{CG1, CGJQ,  CGQ}. For further discussions on the geometric aspects and the AdS/CFT correspondence, see also \cite{An3, CQY, Qi}.

In this paper, we primarily focus on the case where $\partial X$ is non-compact. While the local theory is analogous to the compact boundary case, the global theory is far less developed. Notably, the standard hyperbolic space has two well-known models: the ball model and the upper half-plane model. In the ball model, the conformal infinity is the standard sphere, which is compact. In contrast, in the upper half-plane model, the natural conformal infinity is the  flat Euclidean space, which is non-compact. This distinction makes the study of non-compact boundaries a natural and important extension for generalizing the geometric and analytic theory developed for the standard upper half-plane model.

Beyond the fact that PE manifolds serve as natural generalizations of the upper half-plane model of hyperbolic space, their relationship to PE manifolds with compact boundaries can be viewed in two ways. One motivation for this study is that Poincaré-Einstein (PE) manifolds with non-compact boundaries frequently arise as blow-up limits of those with compact boundaries. Consequently, understanding the non-compact boundary case is critical for gaining insight into the geometry and moduli space of compact cases. For example, in works such as \cite{CG1, CGQ, CGJQ}, sophisticated blow-up arguments were used to control curvature near the boundary to establish compactness properties for Poincaré-Einstein manifolds.

Second, there is a strong connection between PE manifolds with non-compact boundaries and those with singular conformal infinities.  For examples of PE manifolds with singularities, we refer the reader to \cite{AOS, BL}. Specifically, by applying stereographic projection and performing a blow-up of the singular set on the boundary, one obtains PE manifolds with non-compact boundaries and a particular decay order of curvature. Note that the curvature decay in such cases is typically worse than that of manifolds derived from smooth Poincaré-Einstein manifolds with compact boundaries via stereographic projection. Thus, a deeper understanding of general Poincaré-Einstein manifolds with non-compact boundaries provides valuable insights into PE manifolds with compact boundaries that may include singular sets.

The goal of this paper is twofold. The first is to prove a rigidity theorem for Poincaré-Einstein (PE) manifolds with flat Euclidean space as their conformal infinity under certain natural geometric assumptions. Specifically, we show that if a smooth PE manifold with flat Euclidean conformal infinity admits an \emph{adapted boundary defining function} $\rho$--a suitable conformal factor that fixes the conformal gauge symmetry among boundary defining functions--and if the compactified metric $\rho^2g_+$ has non-negative scalar curvature and its full curvature tensor exhibits sufficiently fast quadratic decay, then the manifold must be the standard hyperbolic space.

The second goal is to provide a variety of examples of smooth Poincaré-Einstein manifolds with non-compact boundaries and specific curvature decay properties. These examples arise from the stereographic projection of Poincaré-Einstein manifolds with singular boundaries. Thus, we confirm that the class of PE manifolds with non-compact boundaries is rich and abundant.


Before presenting the main results of this paper, we briefly discuss the gauge freedom of boundary defining functions. In the definition of PE manifolds, the conformal infinity becomes observable by compactifying the metric, which is achieved by multiplying the original metric by a suitable boundary defining function. Of course, when the boundary is topologically non-compact, as in our case, the manifold does not become compact in the usual sense after this multiplication. However, we will continue to use the term compactification for convenience, as it effectively conveys the concept in this context, and we could not identify a better alternative.

Even after selecting a specific representative within the conformal equivalence class of metrics on the boundary, there are still numerous choices for the boundary defining function. To address this ambiguity, it is often necessary to identify a special boundary defining function. For compact boundaries, a commonly used option is the \emph{adapted boundary defining function}, as studied in the work of Case and Chang \cite{CC}, where they introduced a defining function that depends on a real parameter $\gamma$. This approach generalizes the defining functions used in the works of Lee \cite{Le1} and Fefferman and Graham \cite{FG}, which correspond to the specific cases $\gamma = \frac{n}{2}$ and $\gamma = \frac{n}{2} +1$, respectively. Roughly speaking, this defining function solves an elliptic PDE whose coefficients depend on $\gamma$ in the interior and exhibits suitable boundary behavior.

When the boundary is non-compact, we define the adapted boundary defining function in an analogous manner, and assume their existence in our results. We note that the existence of these defining functions is open in general, except in cases where a PE manifold is a limit of PE manifolds with compact boundaries and adapted defining functions. For a precise definition and basic properties of these functions, we refer the reader to Section 2.

Now, we present the main results of this paper. Our first result is a rigidity theorem for Poincaré-Einstein (PE) manifolds with flat Euclidean space $(\mathbb{R}^n, g_0)$ as their conformal infinity:

\begin{theorem} 
Let $n \ge 3$. There exists a positive number $C(n)$ with the following significance:

Suppose $(X^{n+1}, \mathbb{R}^n, g_+)$ is a complete $C^{3,\alpha}$(where $0<\alpha<1$) Poincar\'{e}-Einstein manifold with conformal infinity $(\oR, g_{0})$. Assume there exists the Fefferman-Graham compactification $\rho \in C^3(\overline{X})$ and let $g = \rho^2 g_+ $ be a $C^3$ metric on $\overline{X}$.
Further, assume the following conditions hold:

(1) $R_g\ge 0$

(2) $|Rm_g|(x) \le \frac{C(n)}{\mathrm{dist}(x,o)^2}$ for all $x\in \overline{X}$, where $o$ is a fixed point on the boundary.

Then,  $(\overline{X}, \mathbb{R}^n, \rho^2 g_+)$ is isometric to the standard Euclidean upper half-plane, or equivalently $(X^{n+1}, \mathbb{R}^n,  g_+)$ is isometric to the standard upper-half plane model of hyperbolic space.
\end{theorem}

In fact, our method establishes the rigidity result under the assumption of the existence of an adapted boundary defining function, where the parameter $\gamma$ lies within a broader range. For the geometric significance of the Fefferman-Graham compactification and its connection to $Q$ -curvature and renormalized volume, we refer the reader to \cite{An4, FG, Gr, GZ} and the references therein.

\begin{theorem} 
Let $n \ge 3$ and $\gamma > \mathrm{max}(\frac{n}{2}-1, 1)$. There exists a positive number $C(n, \gamma)$ with the following significance:

Suppose $(X^{n+1}, \mathbb{R}^n, g_+)$ is a complete, smooth Poincar\'{e}-Einstein manifold with conformal infinity $(\oR, g_{0})$. Assume there exists the adapted compactification with parameter $\gamma$, $\rho \in C^3(\overline{X})$, and let $g = \rho^2 g_+$ be the compactified metric on $\overline{X}$.
Further, assume the following conditions hold:

(1) $R_g\ge 0$

(2) $|Rm_g|(x) \le \frac{C(n, \gamma)}{\mathrm{dist}(x,o)^2}$ for all $x\in \overline{X}$, where $o$ is a fixed point on the boundary.

Then,  $(\overline{X}, \mathbb{R}^n, \rho^2 g_+)$ is isometric to the standard Euclidean upper half-plane. Equivalently $(X^{n+1}, \mathbb{R}^n,  g_+)$ is isometric to the standard upper-half plane model of hyperbolic space.

\end{theorem}
We note that in the theorem above, $C^{3,\alpha}$ regularity of the metric is sufficient when $\gamma \ge \frac{3}{2}$. However, when, $1<\gamma<\frac{3}{2}$, a smoothness assumption is required to fully utilize the asymptotic expansion of $\rho$ near the boundary. For further details, we refer readers to Lemma \ref{lem:adpreg1} and Lemma \ref{lem:adpreg3}.

In the final part of the paper, we present several explicit examples of Riemannian manifolds with non-compact boundaries that are conformal to Poincaré-Einstein manifolds in the interior and exhibit quadratic curvature decay at infinity.

These examples arise from stereographic projections of various conformally compact Einstein manifolds in the classical sense, including smooth, non-smooth, and singular cases. The non-smooth example, originally constructed by Bahuaud and Lee \cite{BL}, admits a $C^{1,1}$ conformal compactification in the classical sense. The singular example belongs to the Plebański-Demiański family and was shown by Alvarado, Ozuch, and Santiago \cite{AOS} to have a conic-edge singularity in its conformal compactification.

\subsection{Related results and method of proof}
The rigidity theorem for Poincar\'{e}-Einstein manifolds with regular conformal compactification in the classical sense is well known, i.e.,   if the conformal infinity is $(\mathbb{S}^n, [g_{\mathbb{S}}])$,  then $(X^{n+1}, g_+)$ must be isometric to the hyperbolic space $\mathbb{H}^{n+1}$. This was first proved  by Shi and Tian \cite{ST}  under extra assumptions on the topology and curvature  deday. Later Qing  \cite{Qi1} provided another proof using the positive mass theorem (PMT). The rigidity problem was subsequently resolved without extra assumptions by Dutta and Javaheri \cite{DJ},  and Li, Qing and Shi \cite{LQS}, who derived a lower bound for the Bishop-Gromov volume ratio using the Yamabe constant at conformal infinity. Later, an alternative proof that relied solely on the relationship between conformal invariants of $(X^{n+1}, g_+)$ and $(\mathbb{S}^n, [g_{\mathbb{S}}])$ was given by Chen, Lai and Wang   \cite{CLW1},  and Wang and Wang  \cite {WW}. 

Several other uniqueness theorems exist for cases where the conformal infinity is compact but not the standard sphere. For instance, when the conformal infinity is sufficiently close to the round sphere, Chang, Ge, Jin and Qing \cite{CGJQ} proved that Graham-Lee's small perturbation solution from \cite{GL} is unique. More recently, in a preprint \cite{CYZ}, Chang, Yang and Zhang demonstrated that if the conformal infinity is $\mathbb{S}^1(\lambda)\times \mathbb{S}^2$ with a sufficiently large circle radius $\lambda$, then the Poincar\'{e}-Einstein fill-in is unique. We also note that uniqueness is established when the conformal infinity is homogeneous, as shown in \cite{Li1, Li2}.

In a recent preprint work \cite{CG2}, Chang-Ge proved that if the the conformal infinity is the standard  Euclidean space $\mathbb{R}^3$ with some further  assumptions including the Ricci curvature decay and positive injectivity radius on the scalar flat ($\gamma =\tfrac{1}{2}$ case) conformal compactification, then  $(X^{4}, g_+)$ is isometric to the hyperbolic space $\mathbb{H}^{4}$.  In this paper, we consider the generalized conformal compactification for Poincar\'{e}-Einstein manifold $(X^{n+1}, g_+)$ such that the conformal infinity is not compact. Without extra assumption, the rigidity and uniqueness problem is much harder than the classical case, since we lose the control of infinity.

Our proof relies on the fact that the adapted boundary compactification serves as a regularized distance function from the boundary. In the study of Riemannian manifolds with non-negative Ricci curvature, as shown in \cite{C} and \cite{CM3}, the function $ b= G^{\frac{1}{2-n}}$, where $G$ is the Green's function for the Laplacian, acts as a regularized distance function from a fixed point. Colding and Minicozzi established important monotonicity formulas for the $|\nabla b|$-weighted area and volume, which played a crucial role in proving the uniqueness of tangent cones for Einstein metrics in \cite{CM4}.

There are notable analogies between our adapted boundary defining function $\rho$ and the function $b$ introduced by Colding and Minicozzi. First, both satisfy $|\nabla \rho |\le 1$ and $|\nabla b|\le1$ when the manifold has non-negative scalar curvature and non-negative Ricci curvature, respectively. Moreover, if equality holds at any point (in the interior of Poincaré-Einstein manifolds), the manifold must be flat Euclidean space by the strong maximum principle. Second, similar to the case of manifolds with non-negative Ricci curvature, we observe a comparable monotonicity property for the $|\nabla \rho|$-weighted area functional (see Lemma \ref{lemma:areabound}).

In the work of Colding and Minicozzi \cite{CM1}, \cite{CM2}, they studied the space of harmonic functions on Riemannian manifolds with at most polynomial growth at infinity, leveraging the frequency formula. In our case, we compute the  $|\nabla \rho|$-weighted $L^{2N}$-norm over the level sets of $\rho$ for sufficiently large integer $M$. The curvature tensor satisfies a specific elliptic system, such as the Bach-flat equation in the case where $n+1 =4 $. 

In the case of Einstein manifolds, Bando, Kasue, and Nakajima \cite{BKN} used the Moser iteration technique on the elliptic differential inequality $\Delta |Rm| \ge - C |Rm|^2$ to study the behavior of Einstein manifolds both at infinity and near isolated singularities. This well-known technique has been applied in various contexts, such as  \cite{CT} for Ricci-flat manifolds, \cite{TV1} and \cite{TV2} for Bach-flat manifolds in dimension 4, and \cite{AV} for more general obstruction-flat manifolds in higher dimensions.

Instead of the Moser iteration technique, we use the monotonicity arguments developed in \cite{L, STV1, STV2}, where Liouville-type theorems for subharmonic functions and certain degenerate elliptic PDEs defined on the upper half-space were proven. We generalize these arguments to the case of the Schr\"odinger operator  $\Delta u = f u$ on the flat upper half-plane, where $f$ and $u$ are functions with bounded growth (see Proposition \ref{prop:liouville}). Simply put, if the curvature exhibits sufficiently fast quadratic decay, the curvature equation coupled with $\rho$ implies that the $|\nabla \rho|$-weighted $L^{2N}$-norm over the level sets of $\rho$ exhibits a monotonicity property, which allows us to conclude that the curvature is flat.

Our method applies to a broad range of the parameter $\gamma$ and dimensions $n$. Specifically, we cover $\gamma >\max (1, \frac{n}{2}-1)$, which includes the compactifications of Fefferman-Graham and Lee in the Poincaré-Einstein setting. This is in contrast to the constant scalar curvature case studied in \cite{AV, TV1, TV2}, or the  $\gamma =\tfrac{3}{2}$ (flat $Q$-curvature) case studied in \cite{CG1, CGJQ, CGQ} when $n\geq 5$ . Furthermore, it is worth noting that our assumptions do not necessarily imply the finiteness of $\int_X |Rm|^{\frac{n+1}{2}} dV_g < \infty$.

We are able to bypass the difficulty arising from the fact that the elliptic PDE for scalar curvature is degenerate (see Lemma \ref{lemma:eqcurvature}) and that the curvature equation (obstruction-flatness) is a higher-order elliptic PDE in higher dimensions. Additionally, we establish and apply a curved version of Hardy's inequality (Proposition \ref{prop:hardy}) to address the issue that the adapted metric is not necessarily smooth.

\subsection{Organization of the paper}

In Section 2, we review some preliminary facts from geometric scattering theory and discuss the curvature equations associated with the adapted boundary compactification of Poincaré-Einstein manifolds. In Section 3, under the assumptions of non-negative scalar curvature and sufficiently fast quadratic curvature decay, we establish various estimates for the metric and compute the integration of curvature tensors over the level sets of the adapted boundary defining function. In Section 4, we derive an ordinary differential inequality that describes the propagation of curvature along the level sets. In Section 5, we prove the rigidity theorems. Finally, in Section 6, we provide non-trivial examples of Poincaré-Einstein manifolds with non-compact conformal infinities. Readers are advised to read Sections 4.1 and 5.1 before Sections 4.2 and 5.2, as the case $n\geq4$ equires additional algebraic computations.

\section*{Acknowledgement}
The first author was supported by KIAS Individual Grant, with grant number: MG096101.
The second author was supported by National Natural Science Foundation of China No. 12271348 and Shanghai Science and Technology Innovation Action Plan No. 20JC1413100.

\section{Preliminaries}

\subsection{Geodesic normal defining function}
Assume $(X^{n+1}, g_+)$ is a complete Poinca\'{e}-Einstein metric of  $C^{k,\alpha}$ regularity ($k\geq 2, 0\leq \alpha<1$) with non-compact conformal infinity $(M, [h])$. 
Fix a boundary representation $h$. 
Then  $y\in C^1(\overline{X})$ is called the \textit{geodesic normal defining function} associated to $(g_+, h)$  if it is a boundary defining function and satisfies that $y^2g_+|_{M}=h$ and $|\nabla y|^2_{y^2g_+}\equiv 1$ in a neighborhood of $M$. 

\begin{lemma}\label{lem:geo}
If $x$ is a smooth boundary defining function such that $\bar{g}=x^2g_+\in C^{k,\alpha}(\overline{X})$ and $ h=g|_{M}$, then  there exists a geodesic normal defining function $y$ associated to $(g_+,h)$ in a neighborhood $U \subset X$ of $M$,  such that
\begin{itemize}
	\item[(1)] $y=x+O(x^2)$; 
	\item[(2)] $g=y^2g_+$ has a $C^{k-1,\alpha}$ extension to $\overline{X}$; 
	\item[(3)] for each $p\in M$, there exist $\epsilon_p>0$,  a boundary neighborhood $V_p\subset M$ of $p$ and  a $C^{k-1,\alpha}$ open map : $\Phi: V_p\times [0,\epsilon_p)\rightarrow U$ such that
	$$	\Phi^*g=
	\begin{cases}
	\displaystyle    y^{-2}\left(dy^2+h+x^2h_2+\cdots +	y^nh_n +O(y^{n+1})  \right),  \ &\textrm{for $n$ odd}, 
		\\
	\displaystyle 	y^{-2}\left(dy^2+h+x^2h_2+\cdots +y_n(\log y)O_n+	y^nh_n  +O(y^{n+1}\log y)\right), \  &\textrm{for $n$ even},
	\end{cases}
	$$
	where $h_2, \cdots, h_{n-1}$ and $O_n$ are locally determined by $h$ and $h_n$ is the global term. In particular, 
	 $h_2=-A_h$ with $A_h$ the Schouten tensor of $h$; $O_n$ is the obstruction tensor. 
\end{itemize}	
\end{lemma}
\begin{proof}

The existence and regularity theories  for geodesic normal defining function are according to the proof of \cite[Lemma 5.1]{Le1}. 
	The Taylor expansion of metric $g$ is from the boundary regularity theorem of \cite{CDLS}. 
\end{proof}

\begin{lemma}\label{lem:geocur}
If $y\in C^2(\overline{X})$ is a geodesic normal defining function in a neighborhood $U$ of $M$, then we have the curvature identities: 
\begin{itemize}
	\item[(1)] $ \Delta y = -\frac{1}{2n}yR$ in $U$;
	\item[(2)] $ Ric(\nabla y, \nabla y) = \frac{1}{2n} R$ in $U$;
	\item[(3)] $ R = \frac{n}{n-1} R_h$ on $M$;
	\item[(4)] $(M,h)$ is totally geodesic in $(\overline{X},g)$. 
\end{itemize}	
\end{lemma}
\begin{proof}

Since the curvatures in (1)-(4) are also locally determined, the computations are the same as \cite[Lemma 5.2]{GL}. 
\end{proof}

Notice that here $M$ is non-compact and hence we can not find a uniform  $\epsilon>0$ such that corresponding to every boundary point $p$,  $y$ exists in the interval $[0,\epsilon)$.   Actually, the range for the existence of $y$ has more deep implication.

\begin{proposition}\label{prop:globgeo}
Let $(X^{n+1}, g_+)$ be a complete Poincar\'{e}-Einstein manifold with conformal infinity $(\oR, g_{0})$. 
Suppose there exists a globally defined geodesic defining function $y \in C^{3}(\overline{X})$ such that $g=y^2g_+\in C^{3}(\overline{X})$. Then, $(X^{n+1}, g_+)$ is isometric to the standard upper-half plane model of hyperbolic space.
\end{proposition}

\begin{proof}
Let $h=g|_{\partial X}$. Applying the Laplacian to the equation $|\nabla y |^2 =1 $ and using Bochner’s formula along with the curvature identities in Lemma \ref{lem:geocur}, we obtain: 
\begin{align*}
0 = \frac{1}{2}\Delta(|\nabla y|^2) &= |\nabla^2 y|^2 + \nabla y \cdot \nabla(\Delta y) + Ric(\nabla y, \nabla y)\\
& = |\nabla^2 y|^2 + \partial_y (-\frac{1}{2n}yR) + \frac{1}{2n} R \\
& \ge \frac{1}{n+1} (\Delta y)^2 -\frac{1}{2n} y \partial_y R \\
& = \frac{1}{4n^2(n+1)} y^2 R^2 -\frac{1}{2n} y \partial_y R.
\end{align*}
From this inequality, we deduce that $\partial_y R \ge 0$. Given that the scalar curvature $R_h$ of the boundary metric is zero, it follows that the scalar curvature $R$ is non-negative. 

Now, suppose there exists a point $z_0 = (x_0, y_0)$ where  $R>0$.  Then, along the geodesic starting from $(x_0, 0)$, we have the following differential inequality:
$$ \frac{1}{R^2} \partial_y R \ge \frac{1}{2n(n+1)} y \text{ for $y\ge y_0$}. $$
Integrating both sides yields:
$$ 
\frac{1}{4n(n+1)}(y^2 - y_0^2) \le -\frac{1}{R(x_0, y)} + \frac{1}{R(x_0, y_0)} < \frac{1}{R(x_0, y_0)}. 
$$
Taking the limit $y \rightarrow \infty$ leads to a contradiction.

Thus, we conclude that $R \equiv 0$, and it follows that $\nabla^2 y \equiv 0$. Consequently, the metric $g = y^2 g_+$ is isometric to the Euclidean upper half-space by the Hessian rigidity theorem.
\end{proof}

\subsection{Adapted boundary defining function}
Assume $(X^{n+1}, g_+)$ is a complete Poinca\'{e}-Einstein manifold of   $C^{k,\alpha}$ regularity ($k\geq 2, 0\leq \alpha<1$)   with non-compact conformal infinity $(\partial X, [\hat{g}])$. 
We consider the adapted defining function for $g_+$ in a similar way as the compact model.

Let $s>\frac{n}{2}$ be a real parameter and denote $s = \frac{n}{2} +\gamma$ for $\gamma > 0$.
Then  $\rho \in C^1(\overline{X})$ is called an \textit{adapted boundary defining function} if it is a boundary defining function and satisfies
\begin{itemize}
\item[(i)] 	either $\rho=v^{\frac{1}{n-s}}$ for $s>\frac{n}{2}, s\neq n$, where $v$ is a positive function satisfy the  equation:
\begin{equation}\label{eq.adp1}
-\Delta_+  v - s(n-s)v = 0; 
\end{equation}

\item[(ii)] or $\rho =e^w$ for $s=n$, where $w$ satisfy the   equation
\begin{equation}\label{eq.adp2}
 	-\Delta_+ w=n. 
 \end{equation}
\end{itemize}
Then $g=\rho^2g_+$ is called the \textit{adapted conformal compactification} of $g_+.$ In the following, we fix some $s=\frac{n}{2}+\gamma>\frac{n}{2}$.

The existence theory for adapted boundary defining function relies on the solvability of PDEs (\ref{eq.adp1}) and (\ref{eq.adp2}), however which is not well established for the case of non-compact conformal infinity. For the case of conformally compact Einstein  the spectrum analysis of $\Delta_+$ and Fredholm theorem for  $\Delta_++s(n-s)$ can be refered to \cite{MM, Le2}.

 If exists, then it is obviously that $\rho\in C^{k+1, \alpha}(X)$ by the classical elliptic theory in the interior. 
The regularity of $\rho$ up to the boundary requires a bit more attention. 
There are two ways to observe it.

\begin{lemma}\label{lem:adpreg1}
Suppose $x$ is a smooth boundary defining function such that $\bar{g}=x^2g_+\in C^{k,\alpha}(\overline{X})$. If there exists an adapted  defining function $\rho =x(1+\phi)$ with $\phi=o(1)$, then  $\phi \in C^{l,\beta}(\overline{X})$ and $g=\rho^2g_+\in C^{l,\beta}(\overline{X})$, 
where
\begin{itemize}
	\item[(1)] $l+\beta=k+\alpha$ if $k+\alpha<2\gamma$;
	\item[(2)] $l+\beta=k+\alpha$ if $k+\alpha\geq 2\gamma$ and $2\gamma$ is some positive odd integer;
	\item[(3)] $l+\beta<2\gamma$ if $k+\alpha\geq 2\gamma$ and $2\gamma$ is not a positive odd integer. 
\end{itemize}
\end{lemma}
\begin{proof}
This is from \cite[Lemma 4.6-4.7, 4.9-4.10]{WZ}
\end{proof}

\begin{lemma}\label{lem:adpreg3}
	Suppose $x$ is a smooth boundary defining function such that $\bar{g}=x^2g_+\in C^{\infty}(\overline{X})$, which is also geodesic normal in a neighborhood of $M$. If there exists an adapted  defining function $\rho =x(1+\phi)$ with $\phi=o(1)$, then 
	\begin{itemize}
		\item[(1)] if $\gamma>0$  and $\gamma$ is not an integer, we have $\phi \in C^{l,\beta}(\overline{X})$ with $l+\beta=2\gamma$ and 
		$$
		\rho=x\left(F+x^{2\gamma }G\right), \quad F,G\in C^{\infty}(\overline{X})
		$$
		satisfying $F|_{x=0}=1$, $\partial_xF|_{x=0}=0$; 
		\item[(2)] if $\gamma=k$ is a positive integer, we have $\phi \in C^{k-1,1-\epsilon}(\overline{X})$ and
		$$
		\rho=x(F+x^{k}\log x G), \quad F, G\in C^{\infty}(\overline{X})
		$$
		satisfying $F|_{x=0}=1$, $\partial_xF|_{x=0}=0$. 
	\end{itemize}
\end{lemma}
\begin{proof}
	Since this is pure local theory, it follows the same as the conformally compact case. When the background metric is $C^{\infty}$ conformally compact, the micro-local analysis method applies and gives a full expansion of the solutions to (\ref{eq.adp1}) or (\ref{eq.adp2}), see \cite[Lemma 6.13 and Proposition 7.1]{MM},  \cite[Proposition 3.1]{GZ} and \cite{Gu}.
\end{proof}

\begin{lemma}\label{lem:adpreg2}
If $\rho\in C^{1}(\overline{X})$ is an adapted boundary defining function such that $g=\rho^2g_+\in C^{l,\beta}(\overline{X})$, $l\geq 2, 0\leq \beta<1$, then
 $\rho\in C^{l,\beta}(\overline{X})$. 
\end{lemma}
\begin{proof}
By the formal calculation for adapted boundary defining function in \cite{CC}, 
we have
$$
\begin{cases}
\Delta \rho  = -\frac{s}{2n(s-\frac{n+1}{2})}R\rho, &\quad \textrm{in $X$}, 
\\
\rho=0, &\quad \textrm{on $M$}. 
\end{cases}
$$
Then the classical elliptic theory implies that $\rho\in C^{l,\beta}(\overline{X})$. 
\end{proof}

Now we record some algebraic identities for curvature quantities.

\begin{lemma}\label{lemma:eqs}
If $g=\rho^2g_+\in C^{2}(\overline{X})$ is an adapted conformal compactification, then we have the following identities: 

(1) $\Delta \rho  = -s \rho^{-1}( {1-|\nabla \rho|^2} )$;

(2) $R  =2n (s- \frac{n+1}{2}) \cdot  \rho^{-2}\ (1-|\nabla \rho|^2) $;

(3) $Ric  = -(n-1) \rho^{-1}  \thess  \rho +\frac{2n}{n+1} ( {s-\frac{n+1}{2}} ) \cdot \rho^{-2}\ (1-|\nabla \rho|^2)  g$; 

(4)  $E  = -(n-1)\rho^{-1} \thess \rho $. 

Here   $\thess\rho $ denotes the the trace-free part of the Hessian of $\rho$.
\end{lemma}
\begin{proof}
These are directly from \cite[Lemma 3.2, 6.1]{CC}. 
\end{proof}

\begin{lemma} \label{lemma:eqcurvature}
If $g=\rho^2g_+$ is an adapted conformal compactification, then we have following identities in $X$:  
	
	(1)
	$
	\Delta R +(n+3-2s)\rho^{-1}\langle d\rho, dR\rangle 
	=-2n(2s-n-1)|A|^2+\frac{n+1}{2n(2s-n-1)}R^2;
	$
	
	
	(2)
	$
	\nabla_i A_{ij}=\frac{1}{2n}\nabla_j R;
	$
	
	(3)  if  $(X,g)$ is Bach flat (which is always true for $n=3$)
	$$
	\Delta A_{ij}=\frac{1}{2n}\nabla ^2 R+A*Rm;
	$$
	
	(4) in general, letting $J=\frac{1}{2n}R$, then
	$$
	\begin{aligned}
	 \Delta A_{ij}
= & J_{ij}-\tfrac{n-3}{2}\rho^{-1}\rho^m\left(A_{mi,j}+A_{mj,i}\right)+(n-3)\rho^{-1}\langle \nabla \rho, \nabla A_{ij}\rangle
\\
& -2R_{imjk}A^{mk} 
+(n-3)(A^2)_{ij}
+|A|^2g_{ij}  
+2JA_{ij}. 
\end{aligned}
	$$
\end{lemma}

\begin{proof}
Here (1) is from \cite[Lemma 3.2, 6.2]{CC}; (2) is by easy computation combining with identities in Lemma \ref{lemma:eqs}; (3) is by the definition of Bach flat; (4) is proved in Appendix A.
\end{proof}

We recall the consequences of the Gauss-Codazzi equation for adapted metrics when the boundary metric is Ricci-flat.

\begin{proposition}\label{boundarycondition}
Let $g=\rho^2g_+$ be an adapted conformal compactification with $\gamma>1$, and assume the Ricci curvature of the boundary metric $g|_M$ vanishes. Then 

(1) the metric  $g$ has a  totally geodesic boundary;

(2) the Ricci curvature  of the metric $g$ vanishes on the boundary.

\end{proposition}

\begin{proof}
See \cite[Lemma 2.6]{CGJQ}.
\end{proof}

We record an algebraic inequality that reflects the structure of Poincaré-Einstein manifolds.

\begin{lemma}\label{lemma:algebraicineq}
Given constants $a_0, a_1, a_2$, there exist constants $b, b_0>0 $ depending on  the $a_i$'s such that the following inequality holds:
$$b \rho^2|\nabla R|^2 - a_0 R^2 |\nabla \rho|^2 - a_1 E^2(\nabla \rho, \nabla \rho) - a_2 RE(\nabla \rho, \nabla \rho)  \ge  -b_0 R^2 |\nabla \rho|^2.$$
\end{lemma}
\begin{proof}

From Lemma \ref{lemma:eqs}-(2), it follows that there exist $c_0>0$ and $c_1$ (depending on $n$ and $s$) such that
\begin{equation*}
  \rho^2 |\nabla R|^2 =  c_0^2 E^2(\nabla \rho, \nabla \rho) + 2c_0 c_1 RE(\nabla \rho, \nabla \rho) + c_1^2 R^2|\nabla \rho|^2
\end{equation*}
Thus, it suffices to show that there exists $b$ and $b_0$ such that
$$4(bc_1^2 + b_0 - a_0)(bc_0^2 -a_1) \ge (2bc_0c_1 -a_2)^2,$$
which is equivalent to 
$$4[(b_0-a_0)c_0^2 +  a_2c_0c_1 - a_1c_1^2]b \ge  4 (b_0-a_0)a_1 + a_2^2 .$$
If $b_0$ is chosen sufficiently large so that $(b_0-a_0)c_0^2 +   a_2c_0c_1 - a_1c_1^2>0$, then $b$ can be chosen sufficiently large to satisfy the above inequality.
\end{proof}

\section{Estimates on the level sets}

Suppose the Poincar\'{e}-Einstein manifold $(X, \mathbb{R}^n, g_+)$ has an  adapted conformal metric $g=\rho^2 g_+$ with parameter $s=\frac{n}{2}+\gamma$, where $$\gamma>1.$$ 
We assume that $(X, \mathbb{R}^n, g_+)$ has $C^{\infty}$ regularity when $\gamma<\frac{3}{2}$ and $C^{3,\alpha}$ regularity (where  $0<\alpha<1$)  when $\gamma \ge \frac{3}{2}$. The estimates in this section for the case $\gamma \ge \frac{3}{2}$ remain valid under the $C^{3,\alpha}$ assumption, as in this case, the adapted metric belongs to $C^{3}(\overline{X})$ by Lemma \ref{lem:adpreg1}. Thus, the computations below are well established.

 Let $o$ be a reference point on the boundary, and denote the distance function from $o$ by $d$.
We begin our analysis with the following simple observation:

\begin{lemma}\label{lemma:basic}
Assume that $R\ge 0$ and $|Rm| \le \frac{C_0}{d^2}$. If $C_0$ is sufficiently small, depending on $n$ and $s$, then  there exists a positive constant $\delta$ such that 
$$\delta \le |\nabla \rho| \le 1$$
everywhere, where $\delta$ depends on $C_0$. Furthermore, we have $\delta \rightarrow 1$ as $C_0 \rightarrow 0$. 
\end{lemma}

\begin{proof}
Lemma \ref{lemma:eqs}-(2), we know that  $R \ge 0$ implies  $|\nabla \rho|\le 1$. Since $\rho \equiv 0$ on $\mathbb{R}^n$, it follows that $\rho \le d$.

Moreover, the curvature bound$|Rm|\le \frac{C}{d^2}$ implies that
$$1-|\nabla \rho|^2 \le C_{n,\gamma} C_0 \frac{\rho^2 }{d^2}\le C_{n,\gamma} C_0$$
for some constant $C_{n,\gamma}$. This proves the assertion.
\end{proof}

Throught Sections 3 to 5, we assume that $(X, \mathbb{R}^n, g_+)$ and its adapted conformal metric $g=\rho^2g_+$ satisfies the following conditions::
\begin{equation}\label{assumptions}
\begin{cases}
(1) & R \ge 0 \\
(2) & |Rm|\le \mathrm{min} \{ \frac{C_0}{d^2}, 1\} \text{ for sufficiently small $C_0$. }
\end{cases}
\end{equation}
These assumptions ensure that $|\nabla \rho|$ has a positive lower bound, meaning that the level sets of $\rho$ foliate $X$ , with each leaf diffeomorphic to $\mathbb{R}^n$. Additionally, since $|Rm|$ is uniformly bounded, we may assume, without loss of generality, that  $|Rm|_{L^{\infty}(\overline{X})} \le 1$ by scaling.

\subsection{Estimates on the metric defined on the level sets}

Since we have $\delta \le |\nabla \rho | \le 1$, we consider the gradient flow $\{ \phi_t \}_{t \ge 0}$ and the  level sets $\{\rho = t \} = \phi_t ( \oR )$.   It follows that the metric can be written as $g = \frac{1}{|\nabla \rho|^2} (d\rho)^2 + g_\rho$ where $g_\rho$ is a one-parameter family of metrics on the level sets $\phi_\rho ( \oR ) \simeq \oR$.

We first provide an estimate on the deviation of $g_\rho$  from the flat Euclidean metric $g_0$  on the boundary.

\begin{proposition} \label{lemma:metric}
If $C_0$ is sufficiently small, depending on $a, n, \gamma$, then there exists a constant $C>0$, independent of $\rho$, such that 
$$C^{-1} (\rho + 1)^{-2a} g_0 \le \phi_\rho^* g_\rho \le C (\rho+1)^{2a} g_0.$$
 Here, $g_0$ denotes the flat Euclidean metric on the boundary.
\end{proposition}

\begin{proof}
Consider any straigt-line segment $\sigma$ on the boundary. It suffices to show that there exists a constant $C>0$ such that $$C^{-1}(t + 1)^{-2a} \mathrm{length}(\sigma) \le \mathrm{length}(\phi_t (\sigma)) \le C(t+1)^{2a} \mathrm{length}(\sigma)$$ holds for every $t$ and $\sigma$.

Consider the surface $\Gamma \coloneqq \cup_{t' \le t} \phi_{t'} ( \sigma)$. Since $\phi_t$ is always a diffeomorphism between the boundary and the level sets, the surface $\Gamma$ is smooth. By integrating along $\Gamma$, we obtain:
\begin{align*}
\int_{\sigma_t} |\nabla \rho|    = \int_{\sigma_t} |\nabla^{\Gamma} \rho| & = l(\sigma)+ \int_{\Gamma} \mathrm{div}^{\Gamma}  (\nabla^{\Gamma} \rho )   \\
& = l(\sigma)+ \int_{\Gamma} \Delta^{\Gamma} \rho  \\
& = l(\sigma)+ \int_{\Gamma} \Delta \rho + H^{\Gamma} \cdot \nabla \rho - \nabla_{\nu_k}\nabla_{\nu_k} \rho \\
& =  l(\sigma)+ \int_{\Gamma} \Delta \rho - \nabla_{\nu_k}\nabla_{\nu_k}\rho, 
\end{align*}
where $\nu_k$ are the normal vectors of $\Gamma$. Applying the co-area formula, we obtain
\begin{align*}
\big| \frac{d}{dt} \int_{\sigma_t} |\nabla \rho|  \big| & =  \big| \int_{\sigma_t} \frac{1}{|\nabla \rho|} \big(\Delta \rho - \nabla_{\nu_k}\nabla_{\nu_k}\rho \big) \big|\\
& \le \int_{\sigma_t} \frac{1}{|\nabla \rho|} \big| \Delta \rho - \nabla_{\nu_k}\nabla_{\nu_k}\rho \big| \\
& \le C \int_{\sigma_t} |\nabla \rho| \rho |Ric| \\
& \le C t |Ric|_{L^\infty(\phi_t(\oR))}\int_{\sigma_t} |\nabla \rho|   .
\end{align*}
Since $|Rm|$ is bounded,  for $t \le 1$  we obtain the inequality:
$$    -  C|Ric|_{L^\infty(X)}  \le \log \big[\frac{\int_{\sigma_t} |\nabla \rho|  }{\int_{\sigma_0} |\nabla \rho|  } \big] \le C|Ric|_{L^\infty(X)} .$$
Here, $C$ depends on $\delta$.
For $t\ge 1$,  using the inequality $|Ric|_{L^\infty(\phi_t(\oR))} \le \frac{C_0}{t^2}$, we get:
$$    -  CC_0\log t  \le \log \big[\frac{\int_{\sigma_t} |\nabla \rho|  }{\int_{\sigma_1} |\nabla \rho|  } \big] \le CC_0\log t. $$
If $C_0$ is sufficiently small, then $\delta$ is close to 1 and the constant $CC_0$ above can be chosen to be smaller than $a$, completing the proof.
\end{proof}

Next, we establish estimates for the length and area with respect to the metric $g$.

\begin{lemma} \label{lemma:distancebound}
Let $a>0$. Assume $C_0$ is sufficiently small, as required in Proposition \ref{lemma:metric}. Then, there exists  a constant $C>0$ such that 
$$d(o, \phi_t(B_{2r}\setminus B_{r})) \ge \min \{C\big(t+(t+1)^{-a}r\big),  Cr^{\frac{1}{a+1}}     \}$$
for every $r >0$.
\end{lemma}

\begin{proof}
Suppose  the $d(o, \phi_t(B_{2r}\setminus B_{r}))$ is achieved by a geodesic $\bar{\sigma}$ from $o$ to some point in $\phi_t(B_{2r}\setminus B_{r})$.

We parametrize $\bar{\sigma}$ as $(\alpha(p), \phi_{\alpha(p)} \sigma(p))$ where $p$ is a real parameter and $\sigma$ is a curve in $\oR$. Let $ t_0  \coloneqq \sup \alpha(p) \ge t$.

Then, we estimate:
\begin{align*} \mathrm{length}(\bar{\sigma})& = \int \bigg[ \frac{1}{|\nabla \rho|^2(\alpha(p)) }\alpha'(p)^2 + \phi_{\alpha(p)}^* g_{\alpha(p)} (\sigma'(p), \sigma'(p))\bigg]^{\tfrac{1}{2}} dp \\
& \ge \frac{1}{2} \int \frac{1}{|\nabla \rho|(\alpha(p)) }|\alpha'(p)|  + \big[\phi_{\alpha(p)}^* g_{\alpha(p)} (\sigma'(p), \sigma'(p)) \big]^{\tfrac{1}{2}} dp \\
& \ge \tfrac{1}{2} t_0 + \frac{1}{2}  \int  \big[ \phi_{\alpha(p)}^* g_{\alpha(p)} (\sigma'(p), \sigma'(p)) \big]^{\tfrac{1}{2}} dp \\
& \ge \tfrac{1}{2} t_0  + C (t_0 +1)^{-a} r
\end{align*}
where at the last line we used the inequality
$$\phi_{\alpha(p)}^* g_{\alpha(p)} \ge C ( \alpha(p)+1)^{-2a}g_0\ge C ( t_0+1)^{-2a}g_0.$$

Since the function $\tfrac{1}{2} t_0  + C (t_0 +1)^{-a} r$ is convex, its minimum occurs at either $t_0 = t$ or at the critical point $t_0+1 = (aCr)^{\frac{1}{a+1}}$, completinig the proof.
\end{proof}

\begin{lemma} \label{lemma:areabound}
Let $\Omega$ be a smooth domain on the boundary $\oR$. Then, there exists a constant $C>0$ such that
$$\mathrm{Area}(\phi_t(\Omega)) \le C \mathrm{Area}(\Omega).$$

\end{lemma}

\begin{proof}
Using the co-area formula, we compute:
\begin{align*}
\int_{\phi_t(\Omega)}|\nabla \rho| & = \int_{\Omega} |\nabla \rho| + \int_{\cup_{t'\le t} \phi_{t'} (\Omega)} \Delta \rho \\
& =  \int_{\Omega} |\nabla \rho| -C \int_{\cup_{t'\le t} \phi_{t'} (\Omega)} \rho R \le  \mathrm{Area}(\Omega).
\end{align*}
where we used the fact that $R \ge 0$ and $|\nabla \rho | \le 1$.
\end{proof}

Finally, we prove that the $L^{2N}$-norm of the curvature over the level sets are uniformly bounded provided that $N$ is sufficiently large.
\begin{proposition}\label{prop:curvaturebound}
Let $N$ be a positive integer such that $2N > \frac{n}{2}$. If $C_0$ is sufficiently small (depending on $N$), then $\oint_{\rho =t}|\nabla \rho| |A|^{2N}$ remains uniformly bounded with respect to $t$. 

\end{proposition}

\begin{proof}
Let $a>0$ be a constant to be determined later. Assume $C_0$ is small enough so that Lemma \ref{lemma:distancebound} applies.

Define $B_r$ as the ball on the boundary. Let  $A_0 \coloneqq B_r$ and $A_i \coloneqq B_{2^i r} \setminus B_{2^{i-1} r}$. Using the fact  that the $\phi_t$ decreases the $|\nabla \rho|$-weighted area, we estimate for $t\ge1$:
\begin{align*}
\oint_{\rho =t}|\nabla \rho| |A|^{2N} & \le \sum_{i= 0}^{\infty} \oint_{\phi_t(A_i)} |A|^{2N} dg_t \\
& \le C\frac{1}{t^{4N}} r^n  + \sum_{i= 1}^{\infty} \oint_{\phi_t(A_i)} |A|^{2N} dg_t .
\end{align*}
By Lemma \ref{lemma:distancebound}, and the inequality $\frac{1}{\min(x, y)} \le \frac{1}{x} +\frac{1}{y}$, we obtain:
\begin{align*}
\oint_{\phi_t(A_i)} |A|^{2N}  dg_t  & \le C2^{ni} r^n \mathrm{dist}(o, \phi_t(A_i))^{-4N}\\
& \le C2^{ni} r^n \big[(2^ir)^{\frac{-4N}{a+1}} + (t+2^i (t+1)^{-a}r ) ^{-4N} \big].
\end{align*}
Choosing $r = t (t+1)^a$, we observe that 
$$\oint_{\phi_t(A_i)} |A|^{2N}  dg_t  \le C2^{(n-4N/(a+1))i} t^{(a+1)n-4N}. $$
Thus, if $a$ is suffciently small and $2N > \frac{n}{2}$, then  $\oint_{\rho =t}|\nabla \rho| |A|^{2N}$ is uniformly bounded with respect to $t$.

For $t \le 1$, we estimate:
\begin{align*}
\oint_{\rho =t}|\nabla \rho| |A|^{2N} & \le \sum_{i= 0}^{\infty} \oint_{\phi_t(A_i)} |A|^{2N} dg_t \\
& \le C |A|_{L^\infty }^{2N} r^n  + \sum_{i= 1}^{\infty} \oint_{\phi_t(A_i)} |A|^{2N} dg_t .
\end{align*}
Again, we compute:
\begin{align*}
\oint_{\phi_t(A_i)} |A|^{2N} dg_t  &\le C2^{ni} r^n \big[(2^ir) ^{\frac{-4N}{a+1}} + (t+2^i (t+1)^{-a}r ) ^{-4N} \big]\\
&\le  C2^{ni} r^n \big[(2^ir) ^{\frac{-4N}{a+1}} + (2^{i-a}r ) ^{-4N} \big] \\
&\le C2^{i(n-4N/(a+1))} r^{n-4N/(a+1)}.
\end{align*}
If $a$ is suffciently small and $2N > \frac{n}{2}$, then
$$\oint_{\rho =t}|\nabla \rho| |A|^{2N}  \le C |A|_{L^\infty }^{2N} r^n  +  C r^{n-4N/(a+1)} \sum_{i=1}^\infty   2^{i(n-4N/(a+1))}. $$
Choosing $r =  ( 1+ |A|_{L^{\infty}})^{-\frac{1}{2}}$ yields the required bound, completing the proof. 
\end{proof}

\subsection{$L^2$-gradient estimate for curvature}

In this subsection, we establish Caccioppoli-type gradient estimates for a suitable power of the Ricci tensor. As a preliminary step, we first prove the following curved version of Hardy's inequality.

\begin{proposition}\label{prop:hardy} (Hardy's inequality)
Let $\gamma>1$. For  $b>1$ and $f \in C_c^{\infty}(X)$, we have
$$\int_{\rho \le p} \rho^{1-2b}|\nabla \rho|^2 f^2+  \rho^{3-2b} Rf^2  + \oint_{\rho = p} \rho^{2-2b} |\nabla \rho| f^2 \le  C_{n, \gamma, b} \int_{\rho \le p} \rho^{3-2b} |\nabla f|^2.$$

\end{proposition}

\begin{proof}

For a real number $\alpha$, we observe:

\begin{align*}
0 \le& \int_{\rho \le p} \rho^{3-2b} \big| \frac{\nabla \rho}{\rho} f + \alpha \nabla f \big|^2 \\
= &  \int_{\rho \le p} \rho^{1-2b}|\nabla \rho|^2 f^2 + \alpha^2 \rho^{3-2b} |\nabla f|^2 + 2\alpha \rho^{2-2b} f \nabla \rho \cdot \nabla f \\
= &  \int_{\rho \le p} \rho^{1-2b}|\nabla \rho|^2 f^2 + \alpha^2 \rho^{3-2b} |\nabla f|^2 - \alpha \mathrm{div}( \rho^{2-2b} \nabla \rho ) f^2 \\
& + \oint_{\rho = p} \alpha \rho^{2-2b} |\nabla \rho| f^2 \\
= &  \int_{\rho \le p} (1+(2b-2)\alpha) \rho^{1-2b}|\nabla \rho|^2 f^2 + \alpha^2 \rho^{3-2b} |\nabla f|^2 + \alpha C_{n, \gamma} \rho^{3-2b} Rf^2 \\
& + \oint_{\rho = p} \alpha \rho^{2-2b} |\nabla \rho| f^2.
\end{align*}
Choosing $\alpha$ to be a sufficiently large negative number ensures the desired inequality holds. This completes the proof.
\end{proof}

\begin{proposition}\label{prop:L2gradientweighted1}
Let $1<\gamma<\tfrac{3}{2}$ and let $\eta \in C_c^{\infty} (\overline{X})$ be a test function. Then  the following estimates hold:
$$\int \frac{1}{2} \rho^{3-2\gamma}  \eta^2 |\nabla R|^2 \le 
\int 2 \rho^{3-2\gamma} R^2 |\nabla \eta|^2 +C \rho^{3-2\gamma} |A|^2 R, $$
and
\begin{align*}
\int \rho^{3-2\gamma} \eta^2 |\nabla A|^2 \le & C_{n,\gamma} \bigg[ \int \rho^{3-2\gamma} |\nabla \eta|^2 |A|^2 + \rho^{2-2\gamma} \eta |\nabla \eta| |A|^2  +  \rho^{3-2\gamma} \eta^2 |\nabla R|^2 \\
&  + \rho^{3-2\gamma} \eta^2  |Rm| |A|^2 \bigg].
\end{align*}
\end{proposition}

\begin{proof}
For the scalar curvature, we use Lemma \ref{lemma:eqcurvature}-(1), which gives:
$$\mathrm{div}(\rho^{3-2\gamma} \nabla R) = \rho^{3-2\gamma} A*A.$$
Here, $A*A$ represents a quadratic term in $A$.

We test this equation against $R \eta^2$. When applying integration by parts, the boundary integral over $\mathbb{R}^n$ vanishes by the expansion $\rho$ given in Lemma \ref{lem:adpreg3} near the boundary. Specifically, since $\rho/x =1+ O(x^{2\gamma})$ near the boundary and the scalar curvature of the boundary metric vanishes, the quadratic term in the expansion of $\rho/x$  also vanishes, as shown in \cite[Lemma 6.2]{CC}.
Thus, since $A=O(x^{2\gamma-2})$, we have
\begin{align*}
\int_{\mathbb{R}^n} \rho^{3-2\gamma} \eta^2 R \nabla_x R &= \int_{x = \epsilon} \rho^{3-2\gamma} \eta^2 O(x^{2\gamma-2}) O(x^{2\gamma-3}) \bigg|_{\epsilon = 0} \\
&  = \int_{x = \epsilon}  \eta^2 O(x^{2\gamma-2}) \bigg|_{\epsilon = 0} = 0,
\end{align*}
which holds since $\gamma>1$. This computation will be used repeatedly in this subsection.	
Thus, we obtain the following inequality:
\begin{align*}
0 = & \int \rho^{3-2\gamma} \eta^2 |\nabla R|^2 +2 \rho^{3-2\gamma} \eta R \nabla R \cdot \nabla \eta + \rho^{3-2\gamma}\eta^2 R A*A \ \\
\ge & \int \frac{1}{2} \rho^{3-2\gamma}  \eta^2 |\nabla R|^2 -2 \rho^{3-2\gamma} R^2 |\nabla \eta|^2 -C \rho^{3-2\gamma} \eta^2 |A|^2 R .
\end{align*}
This proves the first inequality. 

For the Schouten tensor, we use Lemma \ref{lemma:eqcurvature}-(4), which follows similarly:
\begin{align*}
&\int  \rho^{3-2\gamma} \eta^2 A*A*Rm  \\
= &  \int \big( - \Delta A_{ij} + \tfrac{1}{2n} \nabla_i \nabla_j R  -  (n-3) \rho^{-1}\nabla_m \rho A_{im,j}+  (n-3)\rho^{-1}\nabla \rho \cdot \nabla A_{ij} \big) \rho^{3-2\gamma} \eta^2 A_{ij} \\
= &  \int \rho^{3-2\gamma} \eta^2 |\nabla A|^2 + 2\eta \rho^{3-2\gamma} A_{ij} \nabla_k A_{ij} \nabla_k \eta  + (3-2\gamma) \rho^{2-2\gamma}\eta^2 A_{ij} \nabla_k A_{ij} \nabla_k \rho \\
&- \frac{1}{2n}\rho^{3-2\gamma} \nabla_j R \nabla_j A_{ij} \eta^2 -\frac{1}{n} \rho^{3-2\gamma} \eta  \nabla_i R \nabla_j \eta A_{ij} -\frac{(3-2\gamma)}{2n} \rho^{2-2\gamma} \eta^2  \nabla_i R \nabla_j \rho A_{ij}  \\
& -(n-3)\rho^{2-2\gamma} \eta^2 \nabla_m \rho A_{im,j} A_{ij} +\tfrac{(n-3)}{2}\rho^{2-2\gamma} \eta^2 \nabla \rho \cdot \nabla |A|^2 \\
\ge &  \int \frac{1}{2} \rho^{3-2\gamma} \eta^2 |\nabla A|^2 - 4 \rho^{3-2\gamma}  |\nabla \eta|^2 |A|^2 -4 \rho^{3-2\gamma} \eta^2|\nabla R|^2 \\
& +\frac{(n-2\gamma)}{2}\rho^{2-2\gamma}  \eta^2 \nabla |A|^2 \cdot \nabla \rho - \frac{(3-2\gamma)}{2n}\rho^{2-2\gamma} \eta^2  \nabla_i R \nabla_j \rho A_{ij}  \\
& -(n-3)\rho^{2-2\gamma} \eta^2 \nabla_m \rho A_{im,j} A_{ij} \\
= & \int \frac{1}{2} \rho^{3-2\gamma}\eta^2 |\nabla A|^2 - 4 \rho^{3-2\gamma}  |\nabla \eta|^2 |A|^2 -4 \rho^{3-2\gamma} \eta^2|\nabla R|^2   \\
& -(n-2\gamma)\rho^{2-2\gamma} \eta |A|^2  \nabla \eta \cdot \nabla \rho - (n-2\gamma)(1-\gamma) \rho^{1-2\gamma}  \eta^2 |A|^2 |\nabla \rho|^2 \\
& - \frac{(n+2\gamma-6)}{2n}\rho^{2-2\gamma} \eta^2  \nabla_i R \nabla_j \rho A_{ij}  + 2(n-3)\rho^{2-2\gamma} \eta \nabla_m \rho \nabla_j \eta A_{im} A_{ij}    \\
&  + (n-3)(2-2\gamma) \rho^{1-2\gamma} \eta^2 \nabla_m \rho \nabla_j \rho A_{im} A_{ij}  + \rho^{3-2\gamma} \eta^2 A*A*A.
\end{align*}
Note that  $(n-2\gamma)(1-\gamma)<0$ for $n\ge3$ and $1<\gamma<\frac{3}{2}$,  and $\nabla_i R \nabla_j \rho A_{ij}$ is a linear combination of $\rho^{-1} RA(\nabla \rho, \nabla \rho)$ and $\rho^{-1} A^2(\nabla \rho, \nabla \rho)$ by Lemma \ref{lemma:eqs}-(2). Using Lemma \ref{lemma:algebraicineq}, and the fact $|\nabla \rho| \le 1$, we can summarize the above computation as follows:
\begin{align*}
\int \rho^{3-2\gamma} \eta^2 |\nabla A|^2 \le & C_{n,\gamma} \bigg[ \int \rho^{3-2\gamma} |\nabla \eta|^2 |A|^2 + \rho^{2-2\gamma} \eta |\nabla \eta| |A|^2  +  \rho^{3-2\gamma} \eta^2 |\nabla R|^2 \\
& + \rho^{1-2\gamma} \eta^2 R^2 |\nabla \rho|^2 + \rho^{3-2\gamma} \eta^2  |Rm| |A|^2 \bigg] \\
\le & C_{n,\gamma} \bigg[ \int \rho^{3-2\gamma} |\nabla \eta|^2 |A|^2 + \rho^{2-2\gamma} \eta |\nabla \eta| |A|^2  +  \rho^{3-2\gamma} \eta^2 |\nabla R|^2 \\
& +  \rho^{3-2\gamma} \eta^2  |Rm| |A|^2 \bigg]. 
\end{align*}
Here, in the last step, we applied Hardy's inequality (Proposition \ref{prop:hardy}) with $f =\eta R$.
\end{proof}

\begin{proposition} \label{prop:weightedgradient}
Let $1<\gamma<\frac{3}{2}$. For sufficiently large $N>0$ and sufficiently small $C_0$ in Assumption (\ref{assumptions}), depending on $N$, we have:
$$\int_{\rho \le p} \rho^{3-2\gamma} |\nabla R^{N+1}|^2 < \infty, $$
$$\int_{\rho \le p} \rho^{3-2\gamma} |\nabla A|^2 |A|^{2N}< \infty. $$
Moreover, in this case, we also have that  $ \frac{1}{p^{2\gamma-2}} \oint_{\rho = p} |A|^{2N+2}$ is uniformly bounded for $p<1$, as a consequence of Hardy's inequality. If $n=3$, the choice $N = 0$ is sufficient.
\end{proposition}

\begin{proof}

Let $p, r \ge 1$. We define the cutoff function $\eta (z) \coloneqq \eta_1 (\phi_{\rho(z)}^{-1} (z)) \eta_2 (\rho(z))$ or $\eta(x, t) = \eta_1(x) \eta_2 (t)$ in the global coordinate.
$\eta_1$ and $\eta_2$ are smooth cutoff functions defined on $\mathbb{R}^n$ and $[0, \infty)$, respectively, as follows: $\eta_1 (x ) = 1$ on $r \le |x| \le 2r$ and $\eta_1 (x) = 0$ on $|x| \le r/2 $ or $|x| \ge 5r/2 $ and $|\nabla_0 \eta_1 | \le C/r$.  $\eta_2 (t) = 1$ on $t \le p$, $\eta_2 (t ) = 0 $ on $t \ge 2p$ and $|\nabla_0 \eta_2 | \le C/p$. Here the  subscript 0 on the $\nabla_0$  indicates differentiation with respect to the flat metric.

It is straightforward to verify the estimate: 
\begin{align*}
|\nabla \eta|^2  \le 2 \eta_2^2 |\nabla  \eta_1|^2 + 2\eta_1^2 |\nabla \eta_2|^2  &\le C\frac{(t+1)^{2a}}{r^2} \eta_2^2 + C \frac{|\nabla \rho|^2(t)}{p^2} \eta_1^2 \\
& \le C\frac{(t+1)^{2a}}{r^2} \eta_2^2 + C \frac{1}{p^2} \eta_1^2 . 
\end{align*}
Applying Proposition \ref{prop:L2gradientweighted1}, we estimate:
\begin{align*}
\int_{\{\rho \le p\}\cap C(r,2r) } \rho^{3-2\gamma}  |\nabla R|^2 & \le C\int_{\{\rho \le 2p\} \cap C(\tfrac{r}{2}, \tfrac{5r}{2}) } 4 \rho^{3-2\gamma} |\nabla \eta|^2 R^2   + \eta^2 \rho^{3-2\gamma} |A|^3\\
& \le C\int_{\{\rho \le 2p\} \cap C(\tfrac{r}{2}, \tfrac{5r}{2}) } \rho^{3-2\gamma} \big(\frac{(\rho+1)^{2a}}{r^2} + \frac{1}{p^2}  \big) R^2  +  \rho^{3-2\gamma}  |A|^3. 
\end{align*}
From Lemma \ref{lemma:distancebound}, we know that $\mathrm{dist}(o, C(r, 2r)) \ge C r^{1/(a+1)}$ for $r \ge 1$. This leads to:
\begin{align*}
\int_{\{\rho \le p\}\cap C(r,2r) } \rho^{3-2\gamma}  |\nabla R|^2 R^{2N} & \le Cr^{-4N/(a+1)} \int_{\{\rho \le 2p\}\cap C(r,2r) } \rho^{3-2\gamma}  |\nabla R|^2 \\
&\le C r^{-4N/(a+1)}\int_{\{\rho \le 2p\} \cap C(\tfrac{r}{2}, \tfrac{5r}{2}) } \rho^{3-2\gamma} (\rho+1)^{2a}  R^2  +  \rho^{3-2\gamma}  |A|^3 \\
&\le C r^{n -4(N+1)/(a+1)}  \int_0^{2p} \rho^{3-2\gamma} (\rho+1)^{2a} d\rho \\
&\le C r^{n -4(N+1)/(a+1)}, 
\end{align*}
where we used Lemma \ref{lemma:areabound} and the co-area formula in the last step. If $N$ is sufficiently large, then for a sufficiently small $C_0$, the constant $a$ can be chosen small enough (by Lemma \ref{lemma:distancebound}), ensuring:
$$\sum_{i} (2^ir)^{n -4(N+1)/(a+1)} < \infty. $$
Thus, we conclude:
$$\int_{\rho \le p} \rho^{3-2\gamma} |\nabla R^{N+1}|^2 < \infty.$$

In particular, if $n=3$, then we can choose $N=0$.
For the Schouten tensor $A$, the proof follows the same steps, except that we use the fact that $\rho^{2-2\gamma}$ is integrable near $\rho = 0$ since $\gamma < \frac{3}{2}$.
\end{proof}

\begin{proposition}
Let $\gamma \ge \tfrac{3}{2}$ and $\eta \in C_c^{\infty} (\overline{X})$ be a test function. Then we  have the following estimates:
$$ \int  \eta^2 |\nabla R|^2\le C\int R^2 |\nabla \eta|^2 + \rho^{-1}\eta |\nabla \eta\cdot \nabla \rho|  R^2 + \eta^2 |A|^3, $$
and
\begin{align*}
\int  \eta^2 |\nabla A|^2 \le & C_{n,\gamma} \bigg[ \int  |\nabla \eta|^2 |A|^2 + \rho^{-1} \eta |\nabla \eta \cdot \nabla \rho| |A|^2  +  \eta^2 |\nabla R|^2  +  \eta^2  |Rm| |A|^2 \bigg].
\end{align*}

\end{proposition}

\begin{proof}
For the scalar curvature, we use the equation:
$$	\Delta R=-(3-2\gamma)\rho^{-1} \nabla  \rho \cdot \nabla R  + A*A . $$
Testing the above equation with $\eta^2 R$, we obtain:
\begin{align*}
0   =&\int  \big(-\Delta R -(3-2\gamma)\rho^{-1} \nabla \rho \cdot \nabla R + A*A \big)  \eta^2 R \\
=& \int \eta^2 |\nabla R|^2  + 2\eta R \nabla \eta \cdot \nabla R +(3-2\gamma) \eta R^2 \rho^{-1} \nabla \eta \cdot \nabla \rho -\tfrac{(3-2\gamma)}{2} \eta^2 \rho^{-2} |\nabla \rho|^2R^2 \\
& + \eta^2 A*A*A \\
\ge &C \int  \eta^2 |\nabla R|^2 - R^2 |\nabla \eta|^2 - \rho^{-1}\eta |\nabla \eta \cdot \nabla\rho|  R^2 - \eta^2 |A|^3, 
\end{align*}
by Young's inequality and the fact that $\gamma \ge \frac{3}{2}$.

For the Schouten tensor, we have:
\begin{align*}
&\int  \eta^2 A*A*Rm  \\
= &  \int \big( - \Delta A_{ij} + \tfrac{1}{2n} \nabla_i \nabla_j R  -  (n-3) \rho^{-1}\nabla_m \rho A_{im,j}+  (n-3)\rho^{-1}\nabla \rho \cdot \nabla A_{ij} \big)  \eta^2 A_{ij} \\
= &  \int  \eta^2 |\nabla A|^2 + 2\eta A_{ij} \nabla_k A_{ij} \nabla_k \eta  - \frac{1}{2n} \nabla_j R \nabla_j A_{ij} \eta^2 -\frac{1}{n}  \eta  \nabla_i R \nabla_j \eta A_{ij} \\
& -(n-3)\rho^{-1} \eta^2 \nabla_m \rho A_{im,j} A_{ij} +\tfrac{(n-3)}{2}\rho^{-1} \eta^2 \nabla \rho \cdot \nabla |A|^2 \\
\ge &  \int \frac{1}{2}  \eta^2 |\nabla A|^2 - 4   |\nabla \eta|^2 |A|^2 -4 \eta^2|\nabla R|^2 +2(n-3)\rho^{-1} \eta \nabla_j \eta  \nabla_m \rho A_{im} A_{ij} \\
& +(n-3)\rho^{-2} \eta^2 \nabla_m \rho \nabla_j \rho A_{im} A_{ij} +\frac{(n-3)}{2n} \rho^{-1} \eta^2 \nabla_m \rho \nabla_i R A_{im}  \\
& + \tfrac{(n-3)}{2}\rho^{-2} \eta^2 |\nabla \rho|^2 |A|^2 - (n-3) \rho^{-1} \eta \nabla \rho \cdot \nabla \eta   |A|^2 + \eta^2 A*A*A. 
\end{align*}
We use the inequality:
$$\rho^{-1}\eta A^2(\nabla \eta, \nabla \rho) \le \rho^{-2} \eta^2 A^2(\nabla \rho, \nabla \rho) + |\nabla \eta|^2 |A|^2, $$
to handle the term $\rho^{-1} \eta \nabla_j \eta  \nabla_m \rho A_{im} A_{ij}$.
Applying Lemma \ref{lemma:algebraicineq} and Hardy’s inequality completes the proof.
\begin{align*}
\int  \eta^2 |\nabla A|^2 \le & C_{n,\gamma} \bigg[ \int  |\nabla \eta|^2 |A|^2 + \rho^{-1} \eta |\nabla \eta\cdot \nabla \rho| |A|^2  +  \eta^2 |\nabla R|^2  +  \eta^2  |Rm| |A|^2 \bigg].
\end{align*}

\end{proof}

\begin{proposition} \label{prop:weightedgradient2}
Let $ \gamma \ge \frac{3}{2} $. For a sufficiently large positive integer $N$ and a sufficiently small $C_0$ in Assumption (\ref{assumptions}) depending on $N$, we have:
$$\int_{\rho \le p} |\nabla R^{N+1}|^2 < \infty,$$
$$\int_{\rho \le p}  |\nabla A|^2 |A|^{2N}< \infty. $$
Moreover, by Hardy's inequality, we additionally obtain that:
$ \frac{1}{p}  \oint_{\rho = p} |A|^{2N+2}$ is uniformly bounded for $p<1$.
If $n=3$, $N$ can be chosen as 0.

\end{proposition}

\begin{proof}
The proof follows a similar approach to that of Proposition \ref{prop:weightedgradient}. The key difference is handling the term $\rho^{-1} \eta |\nabla \eta \cdot \nabla \rho| |A|^2$.
This is achieved by choosing the test function $\eta$ in such a way that:
$$|\nabla \eta \cdot \nabla \rho| \le C\frac{\rho}{p^2}, $$
where the constant $C$ above is independent of $r$.

Specifically, we define the test function as $\eta(x,\rho) = \eta_1(x/r) \eta_2 (\rho/p)$. 
Then, we obtain:
$$|\nabla \eta \cdot \nabla \rho|(x, \rho)  = \tfrac{1}{p}|\eta_1(x/r) \eta_2'(\rho/p) | |\nabla\rho|^2 \le C\frac{\rho}{p^2},  $$
provided that $\eta_2$ is chosen such that $|\eta_2'(\rho)| \le C \rho$. 
\end{proof}

\begin{corollary} \label{lemma:rhogradient}
Let $\gamma >1 $ and $N$ be a sufficiently large positive integer. Then, for $C_0$ sufficiently small dependingon $n, \gamma, N$, and any $p>0$, we have:
$$\int_{\rho \le p } \rho |\nabla R|^2 R^{2N} < \infty \text{ and } \int_{\rho \le p } \rho |\nabla A|^2 |A|^{2N} < \infty.$$

\end{corollary}

\begin{proof}
The result follows directly from Propositions \ref{prop:weightedgradient} and \ref{prop:weightedgradient2}, since $3-2\gamma < 1$.
\end{proof}

\begin{remark}
Depending on $\gamma$, the above $L^2$-gradient estimates suggest that it might be not true that $|\nabla A| = O(d^{-3})$ as the  $ \int_{B_{2d}\setminus B_{d}} |\nabla A|^2$  may not decay as $\frac{1}{d^2} \int_{B_{2d}\setminus B_{d}} |A|^2$ .
\end{remark}

\section{Propagation of curvature over level sets}

In this section, we study the propagation of $L^{2N}$-norm of curvature over the level sets of $\rho$. To achieve this, we will derive   integral identities involving the $L^{2N}$-norm of curvature over these level sets. Throughout this section, we assume the conditions stated in Assumption (\ref{assumptions}).

\subsection{Integral identities in the dimension $n+1=4$}

In this subsection, we consider the case $n+1=4$ and take advantage of the Bach-flat equation for the Schouten tensor.

\begin{proposition}(n+1=4) \label{prop:derivatives}
Define the functions
$$
F (p) \coloneqq \int_{\rho \le p} \rho^{-1} |A|^2|\nabla\rho|^2 dV_g, 
$$
$$
G (p) \coloneqq \int_{\rho \le p} \rho^{-1} |R|^2|\nabla\rho|^2 dV_g.
$$
Assume that $C_0$ in the assumption (\ref{assumptions})  is sufficiently small. Then, for $p>0$, the following formulas hold:

(1) First derivative of $F$:
$$
F'(p) = \oint_{\rho = p }  \rho^{-1} |A|^2|\nabla\rho|. 
$$

(2) Second derivative of $F$:
\begin{align*}
F''(p) =& \frac{1}{p^2} \int_{\rho \le l }  2\rho  |\nabla A|^2 -\frac{1}{18}\rho |\nabla R|^2 -\frac{1}{3}  A(\nabla R, \nabla \rho)   + \rho A*A*Rm  \\
&+ \oint_{\rho = l} \frac{1}{3} \rho^{-1}  A(\nabla R, \nu) -  C_{n,\gamma}  \frac {R|A|^2}{|\nabla \rho| }. 
\end{align*}

(3) First derivative of $G$:
$$
G'(p) =  \oint_{\rho = q }  \rho^{-1} R^2|\nabla\rho| . 
$$

(4) Second derivative of $G$:
\begin{align*}
G''(p) = \frac{1}{p^2}\int_{\rho \le p}  2\rho  |\nabla R|^2   + \rho A*A*Rm  + \oint_{\rho = p} \frac{1}{\rho^2} (2\gamma-3)R^2 |\nabla \rho| -  C_{n,s }\frac {R^3}{|\nabla \rho| }. 
\end{align*}
\end{proposition}

\begin{proof}

Note that $F$ and $G$ are well-defined for  sufficiently small $C_0$   by Proposition \ref{prop:curvaturebound}. Moreover, recall that $|\nabla \rho|$  has a positive lower bound close to 1 if $C_0$ is small.

(1) follows directly from the co-area formula:
$$
F(p) =  \int_{\rho \le p}  \rho^{-1} |A|^2|\nabla\rho|^2 dV_g = \int_0^p \big(\oint_{\rho = q }  \rho^{-1} |A|^2|\nabla\rho|\big) dq. 
$$

Now we prove (2). Let $\nu \coloneqq \frac{\nabla \rho}{|\nabla \rho|} $, $C_r = \cup_{t \ge 0} \phi_t (B_r)$ and $L_r = \cup_{t \ge 0} \phi_t (\partial B_r )$ be the normal vector, the cylinder and its lateral surface generated by the gradient flow of $\rho$, respectively.  Let $\tau$ be the outward normal vector along the lateral side $L_r$.
Since the normal vector of the lateral side of the cylinder $C_r$ is perpendicular to $\nabla \rho$, we obtain:
\begin{align*}
& \oint_{\{\rho = q\} \cap C_r }  \rho |A|^2|\nabla\rho| - \oint_{\{\rho = p\} \cap C_r } \rho  |A|^2|\nabla\rho|\\
 =& \int_{\{p \le \rho \le q\} \cap C_r} \mathrm{div}\big( \rho |A|^2 \nabla \rho \big) \\ 
=& \int_{\{p \le \rho \le q\} \cap C_r}  \rho (\partial_{\nu} |A|^2) |\nabla \rho| + \rho |A|^2 \Delta \rho +  |A|^2 |\nabla \rho|^2  \\ 
= & \int_{\{p \le \rho \le q\} \cap C_r}  \rho (\partial_{\nu} |A|^2) |\nabla \rho| -  C_{n,\gamma} \rho^2  \frac { R|A|^2}{2} +|A|^2 |\nabla \rho|^2 \\
= & \int_p^q \big( \oint_{\{\rho = l\} \cap C_r} \rho \partial_{\nu} |A|^2 -  C_{n,\gamma} \rho^2 \frac {R|A|^2}{|\nabla \rho| }+  |A|^2 |\nabla \rho| \big) dl.
\end{align*}
Using the Bach-flatness $\Delta A = \frac{1}{6} \nabla^2 R + A*Rm$ and the boundary condition $A = 0$ (Proposition \ref{boundarycondition}), we have:
\begin{align*}
&\oint_{\{\rho = l\} \cap C_r} \rho \partial_{\nu} |A|^2 \\
=& \int_{\{\rho \le l\} \cap C_r} \mathrm{div}( \rho \nabla|A|^2) - \oint_{\{\rho \le l\} \cap L_r} \rho \partial_\tau |A|^2 \\
=& \int_{\{\rho \le l\} \cap C_r}  \rho \Delta|A|^2 + \nabla \rho\cdot \nabla |A|^2 - \oint_{\{\rho \le l\} \cap L_r} \rho \partial_\tau |A|^2 \\
=&  \int_{\{\rho \le l\} \cap C_r}  \rho \Delta|A|^2 -  |A|^2 \Delta \rho  + \oint_{\{\rho = l\} \cap C_r}  |A|^2 |\nabla \rho|  - \oint_{\{\rho \le l\} \cap L_r} \rho \partial_\tau |A|^2  \\
=&  \int_{\{\rho \le l\} \cap C_r}  \rho  \Delta|A|^2 -    C_{n,\gamma} \frac {\rho R|A|^2}{2} + \oint_{\{\rho = l\} \cap C_r}  |A|^2 |\nabla \rho| - \oint_{\{\rho \le l\} \cap L_r} \rho \partial_\tau |A|^2 \\
=&  \int_{\{\rho \le l\} \cap C_r}  2\rho  |\nabla A|^2  + \frac{1}{3} \rho   A_{ij}\nabla_i \nabla_j R + \rho A*A*Rm  \\
&   + \oint_{\{\rho = l\} \cap C_r}  \rho |A|^2 |\nabla \rho| - \oint_{\{\rho \le l\} \cap L_r} \rho \partial_\tau |A|^2 \\
=&  \int_{\{\rho \le l\} \cap C_r}  2\rho  |\nabla A|^2 -\frac{1}{18}\rho |\nabla R|^2 -\frac{1}{3}  A(\nabla R, \nabla \rho)   + \rho A*A*Rm  \\
&   + \oint_{\{\rho = l\} \cap C_r}  |A|^2 |\nabla \rho| + \frac{1}{3} \rho  A(\nabla R, \nu)  - \oint_{\{\rho \le l\} \cap L_r} \rho \partial_\tau |A|^2. 
\end{align*}
Now, we observe that the inequality
\begin{align*}
\big |\int_p^q \big( \oint_{\{\rho \le l\} \cap L_r} \rho \partial_\tau |A|^2 \big) dl \big| \le 2\int_p^q \big( \oint_{\{\rho \le l\} \cap L_r} \rho |\nabla A| |A| \big)  dl \le C q  \int_{\{\rho \le q\} \cap L_r} \rho |\nabla A| |A| 
\end{align*}
holds. Furthermore, by Proposition \ref{prop:curvaturebound} and Corollary \ref{lemma:rhogradient}, together with (\ref{assumptions}) for sufficiently small $C_0$ and Proposition \ref{lemma:metric}, we have
$$\int_0^{\infty} \big( \int_{\{\rho \le q\} \cap L_r} \rho |\nabla A| |A|  \big) dr < \infty.$$
Hence, by Fubini's theorem, there exists  a sequence  $r_i \rightarrow \infty$ such that $\int_p^q \big( \oint_{\{\rho \le l\} \cap L_r} \rho \partial_\tau |A|^2 \big) dl $ evaluated at $r = r_i$ converges to 0.

Since all other integrands appearing in the identity 
$$\int_p^q \big( \oint_{\{\rho = l\} \cap C_r} \rho \partial_{\nu} |A|^2 -  C_{n,\gamma} \rho^2 \frac {R|A|^2}{|\nabla \rho| }+  \rho^{-1}|A|^2 |\nabla \rho| \big) dl$$
belong to $L^1$, we apply the Lebesgue dominated convergence theorem to obtain:	
\begin{align*}
& \oint_{\rho = q  }  \rho |A|^2|\nabla\rho| - \oint_{\rho = p  } \rho  |A|^2|\nabla\rho|\\
= & \int_p^q   \bigg[\int_{\{\rho \le l\} \cap C_r}  2\rho  |\nabla A|^2 -\frac{1}{18}\rho |\nabla R|^2 -\frac{1}{3}  A(\nabla R, \nabla \rho)   + \rho A*A*Rm  \\
&   + \oint_{\rho = l} 2 |A|^2 |\nabla \rho| + \frac{1}{3} \rho  A(\nabla R, \nu) -  C_{n,\gamma} \rho^2 \frac {R|A|^2}{|\nabla \rho| }\bigg] dl. 
\end{align*}
This proves that
\begin{align*}
(p^2 F'(p))' =& \int_{\rho \le p}  2\rho  |\nabla A|^2 -\frac{1}{18}\rho |\nabla R|^2 -\frac{1}{3}  A(\nabla R, \nabla \rho)   + \rho A*A*Rm  \\
&+ \oint_{\rho = p}  2|A|^2 |\nabla \rho| + \frac{1}{3} \rho  A(\nabla R, \nu) -  C_{n,\gamma} \rho^2 \frac {R|A|^2}{|\nabla \rho| }.
\end{align*}%
Thus, (2) is proved.

Similarly, for the scalar curvature $R$, we obtain:
\begin{align*}
& \oint_{\{\rho = q\} \cap C_r }  \rho R^2|\nabla\rho| - \oint_{\{\rho = p\} \cap C_r } \rho  R^2|\nabla\rho|\\
= & \int_p^q \big( \oint_{\{\rho = l\} \cap C_r} \rho \partial_{\nu} R^2 -  C_{n,\gamma} \rho^2 \frac {R^3}{|\nabla \rho| }+ R^2 |\nabla \rho| \big) dl, 
\end{align*}
and
\begin{align*}
\oint_{\rho = p} \rho \partial_{\nu} R^2=  & \int_{\rho \le p} \mathrm{div}(\rho \nabla R^2 ) = \int_{\rho \le p } \rho \Delta R^2 + \nabla \rho \cdot \nabla R^2 \\
= &  \int_{\rho \le p } 2\rho |\nabla R |^2 + 2\rho R\Delta R   + \nabla \rho \cdot \nabla R^2 \\
=& \int_{\rho \le p } 2\rho |\nabla R|^2 + \rho Rm*A*A  +(2\gamma-2)  \nabla \rho \cdot \nabla R^2 \\
=& \int_{\rho \le p } 2\rho |\nabla R|^2 + \rho Rm*A*A  - (2\gamma-2)  \Delta \rho R^2 \\
& + (2\gamma-2) \oint_{\rho = p}  R^2 |\nabla \rho | \\
= &  \int_{\rho \le p}  2\rho |\nabla R|^2 +\rho Rm*A*A +  (2\gamma -2) \oint_{\rho = p}  R^2 |\nabla \rho |.
\end{align*}
Hence, we derive the final differential equation:
\begin{align*}
(p^2 G'(p))' =& \int_{\rho \le p}  2\rho  |\nabla R|^2   + \rho A*A*Rm  + \oint_{\rho = p}  (2\gamma-1)R^2 |\nabla \rho| -  C_{n,\gamma} \rho^2 \frac {R^3}{|\nabla \rho| }. 
\end{align*}
\end{proof}

\subsection{Integral identities in general dimensions}

We closely follow the argument from the previous subsection. The main difference is that, instead of using Bach-flatness, we now rely on Lemma \ref{lemma:eqcurvature}-(4). We begin by establishing the following algebraic inequality.

\begin{lemma}\label{lemma:A^3} 
For the adapted boundary compactification of a Poincar\'{e}-Einstein manifold with parameter $s=\frac{n}{2}+\gamma$, we have: for $\epsilon>0$ small
\begin{align*}
\nabla_m \rho A_{mi} A_{ij} \nabla_j |A|^2 & \ge C_{n,\gamma}\big[ -\frac{1}{\rho}(1+\frac{1}{\epsilon}) R^2 |\nabla \rho|^2 |A|^2 - \rho R^2 |\nabla A|^2  -\epsilon \rho |\nabla A|^2 |A|^2 \\
&  - \frac{ \rho }{\epsilon } |\nabla R|^2 |A|^2 \big].
\end{align*}

\end{lemma}

\begin{proof}
Recall the identity $A_{ij}=\frac{1}{n-1}E_{ij}+\frac{1}{2n(n+1)} R g_{ij}$. By summing the estimates derived below, we obtain the desired result.
$$R^2 \nabla \rho \cdot \nabla |A|^2 \ge R^2 |\nabla \rho| |\nabla |A|^2| \ge -\frac{1}{\rho} R^2 |\nabla \rho|^2 |A|^2 - \rho R^2 |\nabla A|^2,  $$
$$RE(\nabla \rho, \nabla |A|^2) \ge -2 R|\nabla \rho| |\nabla A||A|^2  \ge -\epsilon \rho |\nabla A|^2 |A|^2 - \frac{1}{\epsilon\rho } R^2 |\nabla \rho|^2 |A|^2, $$
\begin{align*}
E^2 (\nabla \rho, \nabla |A|^2) &= 2  \sum_{a,b} \sum_{i, j, k} E_{ik} E_{jk} (A_{ab}\nabla_i \rho) \nabla_j A_{ab}\\
 &\ge - \frac{1}{\epsilon \rho}  \sum_{a,b}\sum_{i, j, k}  E_{ik} E_{jk} (A_{ab}\nabla_i \rho) (A_{ab} \nabla_j \rho) - \epsilon \rho E_{ik}E_{jk} \nabla_i A_{ab}\nabla_j A_{ab} \\
 &\ge - \frac{1}{\epsilon \rho} E^2(\nabla \rho, \nabla \rho) |A|^2 - \epsilon \rho |A|^2 |\nabla A|^2 \\
&\ge - \frac{C_{n,\gamma} \rho }{\epsilon } |\nabla R|^2 |A|^2 - \frac{C_{n,\gamma}}{\epsilon \rho} R^2 |A|^2 |\nabla \rho|^2 - \epsilon \rho |A|^2 |\nabla A|^2, 
\end{align*}
where we applied apply Lemma \ref{lemma:algebraicineq} to handle  $\frac{1}{\rho} E^2(\nabla \rho, \nabla \rho) |A|^2$. 
\end{proof}


\begin{proposition}\label{prop:highdimderivatives}
For non-negative integers $k, l$, define
$$F_{k,l} (p) = \int_{\rho \le p } \rho^{-1} |A|^{2k} R^{2l}  |\nabla \rho|^2 dV_g. $$
Assume $C_0$ in Assumption (\ref{assumptions}) is sufficiently small. Then, for $p>0$, the following formulas hold:

(1) First derivative:
$$F_{k,l}'(p) = \oint_{\rho = p } \rho^{-1} |A|^{2k} R^{2l}  |\nabla \rho| dV_g.  $$

(2) Second derivative: for any $\epsilon_k>0$, 
\begin{align*}
F_{k,l}''(p) \ge & \frac{1}{p^2} \int_{\rho \le p} k(k-1) \rho R^{2l} |A|^{2k-4} |\nabla |A|^2|^2 +k \rho R^{2l}|A|^{2k-2} |\nabla A|^2  \\
&+ l(2l-1)\rho |A|^{2k}  R^{2l-2} |\nabla R|^2   \\
&-k C_{k, l, n, s}  \big[  |A|^{2k-2} R^{2l} \big(\rho |\nabla R|^2 + \frac{1}{\rho} R^2 |\nabla \rho|^2 \big)  +R^{2l+2} |A|^{2k-4}  |\nabla A|^2 \big]\\
& -  klC_{k,l,n,\gamma}  \big[ \epsilon_k |A|^{2k} R^{2l-2} \big(\frac{1}{\rho} |A|^2 |\nabla \rho|^2 +\rho |\nabla A|^2 \big)+\frac{1} {\epsilon_k }\rho  |A|^{2k-2} R^{2l} |\nabla R|^2  \big] \\
& +C_{k,l,n,\gamma} \rho A^{2k+2l}*Rm  \\
&+ \frac{1}{p^2} \oint_{\rho = p } \frac{k}{n}  |A|^{2k-2}  R^{2l}\rho A(\nabla R, \nu) -2k(n-3)   |A|^{2k-2}R^{2l}  \nabla_m \rho A_{mi} A_{ij} \nu_j\\
&+(2\gamma-3)|A|^{2k} R^{2l} |\nabla \rho| + 2(1-\delta_{0k})(n-s)|A|^{2k}R^{2l} |\nabla \rho| \\
&  -C_{k,l,n,\gamma}\frac{\rho^2}{|\nabla \rho|} A^{2k+2l}*Rm. 
\end{align*}

\end{proposition}

\begin{proof}

The proof follows the same approach as in the $n+1=4$ case, but with additional algebraic complexity.

For the scalar curvature, we compute
\begin{align*}
\Delta R^{2l} =& 2l R^{2l-1} \Delta R + 2l(2l-1) R^{2l-2} |\nabla R|^2 \\
=& 	  -2l (n+3-2s)\rho^{-1}\langle d\rho, dR\rangle R^{2l-1}+  2l(2l-1) R^{2l-2} |\nabla R|^2 + Rm*A^{2l}.
\end{align*}
For the Schouten tensor, we have:
\begin{align*}
\Delta |A|^{2k} =& 2k |A|^{2k-2} A_{ij} \Delta A_{ij}  + 2k|A|^{2k-2} |\nabla A|^2  + k(k-1) |A|^{2k-4} |\nabla A^2|^2 \\
=& kI_1 + 2k|A|^{2k-2} |\nabla A|^2  + k(k-1) |A|^{2k-4} |\nabla A^2|^2  + Rm*A^{2k}
\end{align*}
where 
$$I_1 = \big[ \frac{1}{n} \nabla_i \nabla_j R -(n-3)\rho^{-1}\rho^m\left(A_{mi,j}+A_{mj,i}\right)+2(n-3)\rho^{-1}\langle \nabla \rho, \nabla A_{ij}\rangle \big] A_{ij}  |A|^{2k-2}.$$

Using integration by parts, we obtain:
\begin{align*}
& \oint_{\rho = p} \rho \partial_{\nu} (|A|^{2k} R^{2l}) = \int_{\rho \le p} \mathrm{div}( \rho \nabla (|A|^{2k} R^{2l}) ) \\
=& 
\int_{\rho \le p}  \rho \Delta (|A|^{2k} R^{2l}) + \nabla \rho \cdot \nabla (|A|^{2k} R^{2l}) \\
=&
 \int_{\rho \le p}\rho  R^{2l}\Delta |A|^{2k} + 2kl \rho  |A|^{2k-2} R^{2l-2} \nabla |A|^2 \cdot \nabla R^2 + \rho |A|^{2k}\Delta R^{2l}  \\
&
- \Delta  \rho  (|A|^{2k} R^{2l}) + \oint_{\rho = p}  (|A|^{2k} R^{2l}) |\nabla \rho| \\
=&
 \int_{\rho \le p}  k \rho R^{2l}I_1   -2l (n+3-2s) |A|^{2k} R^{2l-1} \nabla \rho \cdot \nabla R  + \rho I_2   - \rho Rm*A^{2k+2l}  \\
&
 + \oint_{\rho = p} (|A|^{2k} R^{2l}) |\nabla \rho| , 
\end{align*}
where 
\begin{align*}
I_2 =&   2kl |A|^{2k-2} R^{2l-2} \nabla |A|^2 \cdot \nabla R^2 +R^{2l}\big[2k|A|^{2k-2} |\nabla A|^2  \\
&+ k(k-1) |A|^{2k-4} |\nabla |A|^2|^2  \big] + 2l(2l-1)|A|^{2k}  R^{2l-2} |\nabla R|^2. 
\end{align*}
We analyze each term systematically by applying integration by parts and algebraic inequalities to express them in the desired form. 

We compute each term of $I_1$:
\begin{align*}
&\int_{\rho \le p} 2k(n-3)  R^{2l}|A|^{2k-2} \nabla_m \rho A_{ij,m} A_{ij} -2l (n+3-2s) |A|^{2k} R^{2l-1} \nabla \rho \cdot \nabla R  \\
= & \int_{\rho \le p}  2k(n-s) R^{2l}|A|^{2k-2} \nabla |A|^2 \cdot \nabla \rho -(n+3-2s)  \nabla (R^{2l}|A|^{2k} ) \cdot \nabla \rho\\
= & \int_{\rho \le p}  2k(n-s) R^{2l}|A|^{2k-2} \nabla |A|^2 \cdot \nabla \rho + C_{k,l,n,\gamma} Rm * A^{2k+2l} + \oint_{\rho = p} (2\gamma-3 ) R^{2l}|A|^{2k} |\nabla \rho| \\
= & \int_{\rho \le p} - 2(n-s) (1-\delta_{0k}) |A|^{2k} \nabla R^{2l} \cdot \nabla \rho + C_{k,l,n,\gamma} Rm * A^{2k+2l} \\
& + \oint_{\rho = p} (2\gamma-3 ) R^{2l}|A|^{2k} |\nabla \rho|  + 2(1-\delta_{0k})(n-s)R^{2l}|A|^{2k} |\nabla \rho| , 
\end{align*}
where $(1-\delta_{0k})$ is 0 if $k=0$. And
\begin{align*}
&
\int_{\rho \le p}  \rho R^{2l}   |A|^{2k-2} A_{ij}\nabla_i \nabla_j R  \\
=&
 \int_{\rho \le p} - \frac{1}{2n} \rho R^{2l}   |A|^{2k-2} |\nabla R|^2  - R^{2l}   |A|^{2k-2} A_{ij} \nabla_j R \cdot \nabla_i \rho  \\
&
-  (k-1)\rho R^{2l}   |A|^{2k-4} A_{ij}\nabla_j R \nabla_i |A|^2 -2l \rho R^{2l-1}   |A|^{2k-2} A_{ij}\nabla_j R \nabla_i R   \\
& + \oint_{\rho = p} \rho R^{2l} |A|^{2k-2} A(\nabla R, \nu) \\
\ge &  \int_{\rho \le p} - \frac{1}{2n} \rho R^{2l}   |A|^{2k-2} |\nabla R|^2  - R^{2l}   |A|^{2k-2} A_{ij} \nabla_j R \cdot \nabla_i \rho  \\
&
-  \epsilon \rho  R^{2l} |A|^{2k-2} |\nabla A|^2 - \frac{C_{k,l}}{\epsilon}  \rho |A|^{2k-2} R^{2l} |\nabla R|^2 - \epsilon \rho |A|^{2k} R^{2l-2} |\nabla R|^2 \\
& + \oint_{\rho = p} \rho R^{2l} |A|^{2k-2} A(\nabla R, \nu)
\end{align*}
using Young's inequality. Next, we have
\begin{align*}
& 
\int_{\rho \le p}  R^{2l}|A|^{2k-2} \nabla_m \rho A_{mi,j} A_{ij} \\
=& \int_{\rho \le p} - R^{2l}|A|^{2k-2} \nabla_j  \nabla_m \rho A_{mi} A_{ij} - \frac{1}{2n}  R^{2l}|A|^{2k-2} \nabla_m \rho A_{mi} \nabla_i R  \\
&
 -2l  R^{2l-1}|A|^{2k-2}  A_{mi} A_{ij} \nabla_j R \nabla_m \rho - (k-1)  R^{2l}|A|^{2k-4} \nabla_m \rho A_{mi} A_{ij} \nabla_j |A|^2 \\
&
 + \oint_{\rho = p }  R^{2l} |A|^{2k-2} \nabla_m \rho A_{mi} A_{ij} \nu_j \\
=& \int_{\rho \le p} C_{k,l,n,\gamma}\rho Rm * A^{2k+2l} - \frac{1}{2n}  R^{2l}|A|^{2k-2} A_{ij} \nabla_i R  \nabla_i \rho \\
&
 -2l  R^{2l-1}|A|^{2k-2}  A_{mi} A_{ij} \nabla_j R \nabla_m \rho - (k-1)  R^{2l}|A|^{2k-4} \nabla_m \rho A_{mi} A_{ij} \nabla_j |A|^2 \\
&
 + \oint_{\rho = p }  R^{2l} |A|^{2k-2} \nabla_m \rho A_{mi} A_{ij} \nu_j.
\end{align*}
Summing the above computations, we obtain:
\begin{align*}
& \int_{\rho \le p}  k\rho R^{2l}I_1 -2l (n+3-2s) |A|^{2k} R^{2l-1} \nabla \rho \cdot \nabla R  \\
\ge & k\int_{\rho \le p}   - \frac{1}{2n^2 } \rho R^{2l}   |A|^{2k-2} |\nabla R|^2  +\frac{n-4}{n} R^{2l}   |A|^{2k-2} A_{ij} \nabla_j R \cdot \nabla_i \rho \\
& -  \epsilon \rho |A|^{2k-2} R^{2l} |\nabla A|^2  - \epsilon \rho  |A|^{2k} R^{2l-2} |\nabla R|^2 -  \frac{C_{k,l}}{\epsilon} \rho |A|^{2k-2} R^{2l} |\nabla R|^2 \\
& +4l(n-3)  R^{2l-1}|A|^{2k-2}  A_{mi} A_{ij} \nabla_j R \nabla_m \rho \\
&+2 (k-1)(n-3)  R^{2l}|A|^{2k-4} \nabla_m \rho A_{mi} A_{ij} \nabla_j |A|^2\\
& - \frac{2(n-s)}{k} (1-\delta_{0k}) |A|^{2k} \nabla R^{2l} \cdot \nabla \rho +C_{k,l,n,\gamma}\rho A^{2k+2l}*Rm  \\
&+ \oint_{\rho = p } \frac{k}{n} \rho R^{2l} |A|^{2k-2} A(\nabla R, \nu) -2k(n-3)  R^{2l} |A|^{2k-2} \nabla_m \rho A_{mi} A_{ij} \nu_j  \\
&+ (1-\delta_{0k})(n-s)R^{2l}|A|^{2k} |\nabla \rho|  + (2\gamma-3 ) R^{2l}|A|^{2k} |\nabla \rho|.
\end{align*}
Choosing $\epsilon$ depending on $k, l, n, s$ and noting that $A(\nabla R, \nabla \rho) \ge -C_{k,l,n,\gamma} \frac{1}{\rho} R^2 |\nabla \rho|^2$, we compute:
\begin{align*}
& \int_{\rho \le p}  k\rho R^{2l}I_1 + \rho I_2 -2l (n+3-2s) |A|^{2k} R^{2l-1} \nabla \rho \cdot \nabla R  \\
\ge & \int_{\rho \le p} 2kl\rho  |A|^{2k-2} R^{2l-2} \nabla |A|^2 \cdot \nabla R^2 - 2(n-s) (1-\delta_{0k}) |A|^{2k} \nabla R^{2l} \cdot \nabla \rho \\
&  + k(k-1) \rho R^{2l} |A|^{2k-4} |\nabla |A|^2|^2   +k\rho R^{2l}|A|^{2k-2} |\nabla A|^2  + l(2l-1) \rho |A|^{2k}  R^{2l-2} |\nabla R|^2  \\
&-k C_{k, l, n, s}  |A|^{2k-2} R^{2l}\big(\rho |\nabla R|^2 + \frac{1}{\rho} R^2 |\nabla \rho|^2 \big)  \\
&  +4lk(n-3)  R^{2l-1}|A|^{2k-2}  A_{mi} A_{ij} \nabla_j R \nabla_m \rho +2 k(k-1)(n-3)  R^{2l}|A|^{2k-4} \nabla_m \rho A_{mi} A_{ij} \nabla_j |A|^2 \\
&    +C_{k,l,n,\gamma} \rho A^{2k+2l}*Rm  \\
&+ \oint_{\rho = p } \frac{k}{n} \rho R^{2l} |A|^{2k-2} A(\nabla R, \nu) -2k(n-3)  R^{2l} |A|^{2k-2} \nabla_m \rho A_{mi} A_{ij} \nu_j \\
&+ (1-\delta_{0k})(n-s)R^{2l}|A|^{2k} |\nabla \rho|  + (2\gamma-3 ) R^{2l}|A|^{2k} |\nabla \rho|. 
\end{align*}
We observate that, by Young's inequality,
\begin{align*}
&2kl |A|^{2k-2} R^{2l-2} \nabla |A|^2 \cdot \nabla R^2 - 2(n-s)(1-\delta_{0k}) |A|^{2k} \nabla R^{2l} \cdot \nabla \rho  \\
\ge & - klC_{k,l,n,\gamma}  \big[ \epsilon_k   R^{2l-2} |A|^{2k}\big(\frac{1}{\rho} |A|^2 |\nabla \rho|^2 +\rho |\nabla A|^2 \big)+\frac{1} {\epsilon_k }\rho R^{2l} |A|^{2k-2} |\nabla R|^2  \big]. 
\end{align*}
For the last two terms in the interior integral, we require additional algebraic computations. First, we have:
\begin{equation*}
RA_{mi} A_{ij} \nabla_j R \nabla_m \rho =  A^2(R\nabla \rho, \nabla R) \ge  -\epsilon_k \frac{1}{\rho}  |A|^4 |\nabla \rho|^2  -\frac{1}{\epsilon_k }\rho R^2 |\nabla R|^2 . 
\end{equation*}
For the $\nabla_m \rho A_{mi} A_{ij} \nabla_j |A|^2$,  we refer to Lemma \ref{lemma:A^3}, which states:
\begin{align*}
\nabla_m \rho A_{mi} A_{ij} \nabla_j |A|^2 & \ge C_{n,\gamma}\big[ -\frac{1}{\rho}(1+\frac{1}{\epsilon}) R^2 |\nabla \rho|^2 |A|^2 - \rho R^2 |\nabla A|^2  -\epsilon \rho |\nabla A|^2 |A|^2 \\
&  - \frac{ \rho }{\epsilon } |\nabla R|^2 |A|^2 \big].
\end{align*}
Choosing $\epsilon$ sufficiently small, we obtain:
\begin{align*}
&  \oint_{\rho = p} \rho \partial_{\nu} (|A|^{2k} R^{2l})  \\
\ge & \int_{\rho \le p} k(k-1) \rho R^{2l} |A|^{2k-4} |\nabla |A|^2|^2 +k \rho R^{2l}|A|^{2k-2} |\nabla A|^2  + l(2l-1)\rho |A|^{2k}  R^{2l-2} |\nabla R|^2   \\
&-k C_{k, l, n, s}  \big[ |A|^{2k-2} R^{2l}\big(\rho |\nabla R|^2 + \frac{1}{\rho} R^2 |\nabla \rho|^2 \big)  +|A|^{2k-4} R^{2l+2} |\nabla A|^2 \big]\\
& -  klC_{k,l,n,\gamma}  \big[\epsilon_k |A|^{2k}  R^{2l-2} \big(\frac{1}{\rho} |A|^2 |\nabla \rho|^2 +\rho |\nabla A|^2 \big)+\frac{1} {\epsilon_k }\rho  |A|^{2k-2} R^{2l} |\nabla R|^2  \big] \\
&    +C_{k,l,n,\gamma} \rho A^{2k+2l}*Rm  \\
&+ \oint_{\rho = p } \frac{k}{n} \rho R^{2l} |A|^{2k-2} A(\nabla R, \nu) -2k(n-3)  R^{2l} |A|^{2k-2} \nabla_m \rho A_{mi} A_{ij} \nu_j \\
&+(2\gamma-2)|A|^{2k} R^{2l} |\nabla \rho| +2 (1-\delta_{0k})(n-s)|A|^{2k}R^{2l} |\nabla \rho| .
\end{align*}
The remainder of the proof follows similarly to Proposition \ref{prop:derivatives}.

\end{proof}


\section{proof of main theorem}

In this section, we prove the main theorems. The key idea is to analyze an ordinary differential inequality for the $L^{2N}$-norm of the curvature over the level sets, which was derived in Section 4. In short, this differential inequality implies that the integral of the curvature over each level set must grow at least polynomially, while our assumption on quadratic curvature decay imposes at most logarithmic growth. The contradiction between these two growth behaviors forces the curvature to vanish everywhere.

This ODE is defined on the interval $(0, \infty)$ and is subject to boundary conditions at $p =0$, where $p$ denotes the height of the level set.  The curvature vanishes at the boundary, and its more precise behavior as $p \rightarrow 0$ is estimated as a consequence of Hardy’s inequality (see Propositions \ref{prop:weightedgradient} and \ref{prop:weightedgradient2}).

Our method leads to the following Liouville-type theorem, which generalizes the Liouville theorem of \cite{L} for subharmonic functions to operators with a potential exhibiting quadratic decay. Notably, the assumption of quadratic decay is natural, as it is preserved under scaling.

\begin{proposition} \label{prop:liouville}
Let $n \ge 3$. Assume $u $ is smooth function on $\R^n$ satisfying the PDE 
$$\Delta_0 u = - fu.$$ Assume that
$$|f| \le \frac{C_1}{r^2}, \quad |u|\le \frac{C_2}{r^\epsilon}$$
 for some $\epsilon, C_1, C_2 >0$.  If $C_1$ is sufficiently small depending on $n$ and $\epsilon$, then $u \equiv 0$.
\end{proposition}

\begin{proof}
Let $B_r$ be a geodesic ball of radius $r$centered at a point on the boundary. Using integration by parts, we obtain the following identity:

\begin{equation} \label{eq1}
0 \le \frac{1}{r^{n-1}} \int_{B_r} |Du|^2 = \frac{1}{r^{n-1}} \int_{\partial B_r} \partial_r u \cdot u + \frac{1}{r^{n-1}} \int_{B_r} fu^2. 
\end{equation}
Next, we estimate the second integral:
$$\int_{B_r} fu^2 \le C_n \int_0^r s^{n-1} \sup_{\partial B_s} |f|\fint_{\partial B_s} u^2 \le C_1C_n \int_0^r s^{n-3} \fint_{\partial B_s} u^2 $$
where $C_n$ is a dimensional constant. Define
$$g(r) =  \int_0^r s^{n-3} \fint_{\partial B_s} u^2.$$ 
Then the inequality (\ref{eq1}) is transformed into the second order ODE:
\begin{equation} \label{eq2}
\frac{(n-3)}{r} g'(r) \le g''(r) + \frac{C_1C_n}{r^2} g(r). 
\end{equation}
Let $x, y$ be the indicial roots of the characteristic equation associated with this differential inequality:
$$x^2 -(n-2)x + C_1C_n = 0.$$
The roots $x, y$ exist if the inequality $(n-2)^2 \ge4C_1 C_n$ holds. Both roots are positive and satisfy $x, y<n-2$.
Now, define $h(r) = \frac{g(r)}{r^x}$. 
Rewriting (\ref{eq2}) in terms of $h(r)$, we obtain:
 $$0 \le (\frac{h'(r)}{r^{y-x-1}})'.$$
Since  $g(r)\sim r^{n-2}$ for $r$ near 0, it follows that $\frac{h'(r)}{r^{y-x-1}}\sim r^{n-2-y}$.
 As $y<n-2$, we conclude that
$$ \lim_{r\rightarrow0} \frac{h'(r)}{r^{y-x-1}} = 0. $$
Thus, we have  $\frac{h'(r)}{r^{y-x-1}} \ge 0$,  which in turn implies $h'(r) \ge 0$, leading to the inequality:
\begin{equation}\label{gmonotone}
g(1) \le \frac{g(r)}{r^x}.
\end{equation}
Now using the assumption $|u| \le \frac{C_0}{r^\epsilon}$, we obtain an upper bound on $g$:
$$g(r) \le \int_0^r s^{n-3}(1+ C/s^{2\epsilon}) \le C(r^{n-2 - 2\epsilon} + 1).$$
On the other hand, if $C_1$ is sufficiently small, then $x>n-2-2\epsilon$. Taking the limit $r \rightarrow \infty$ on the inequality (\ref{gmonotone}), we conclude that $g \equiv 0$, and thus $u \equiv 0$ provided $C_1$ is sufficiently small.
\end{proof}

For the remainder of this section, we focus on deriving a monotonicity formula of the form (\ref{gmonotone}) for the integral of suitable powers of curvature quantities. This will be a key step in proving the main theorems.


\subsection{n+1=4 dimensional case}

We first begin with the following simple algebraic observation.

\begin{lemma}\label{lemma:algebraic}
Let $n=3$. For $\gamma>1$, the constant of $\rho^{-1} E^2(\nabla \rho, \nabla \rho)$ in the expression of $A(\nabla R, \nabla \rho)$ is positive. In particular, $ A(\nabla R, \nabla \rho) + c_\gamma \rho^{-1}R^2|\nabla \rho|^2 \ge 0$ for some $c_\gamma>0$.

\end{lemma}

\begin{proof}

By direct computation, we obtain
\begin{align*}
A(\nabla R, \nabla \rho) = 3(s-2)\rho^{-1} E^2(\nabla \rho, \nabla \rho) +C_\gamma \rho^{-1} RE(\nabla \rho, \nabla \rho) +   C_{\gamma} \rho^{-1} R^2 |\nabla \rho|^2.
\end{align*}
The second assertion follows from Young's inequality.
\end{proof}

With this, we are now ready to prove the main theorem in the case  $n+1 = 4 $.
\begin{theorem} 
Let $n = 3$ and $\gamma > 1$. There exists a positive number $C_0(3, \gamma)$ with the following significance:

Suppose $(X^{4}, \mathbb{R}^3, g_+)$ is a complete, smooth  Poincar\'{e}-Einstein manifold with conformal infinity $(\mathbb{R}^3, g_{0})$. Assume there exists the adapted compactification with parameter $\gamma$, $\rho \in C^3(\overline{X})$, and let $g = \rho^2 g_+$ be the compactified metric.
Further, assume the following conditions hold:

(1) $R_g\ge 0$

(2) $|Rm_g|(x) \le \frac{C(3, \gamma)}{\mathrm{dist}(x,o)^2}$ for all $x\in \overline{X}$, where $o$ is a fixed point on the boundary.

Then,  $(\overline{X}, \mathbb{R}^3, \rho^2 g_+)$ is isometric to the standard Euclidean upper half-plane. Equivalently $(X^{4}, \mathbb{R}^3,  g_+)$ is isometric to the standard upper-half plane model of hyperbolic space.
\end{theorem}

\begin{proof}

We may assume $C_0(3, \gamma)$ is sufficiently small so that the assumptions of Proposition \ref{prop:derivatives} are satisfied. In particular, $|\nabla \rho |$ has a positive lower bound and we have the curvature bound
$$|Rm|\le C_0(3,\gamma)/d^2 \le C_0(3,\gamma)/\rho^2.$$
Define the following functions:
$$
F (p) \coloneqq \int_{\rho \le p} \rho^{-1} |A|^2|\nabla\rho|^2 dV_g
$$
$$
G (p) \coloneqq \int_{\rho \le p} \rho^{-1} |R|^2|\nabla\rho|^2 dV_g.
$$
We now derive a differential inequality for  $F(p) + cG(p)$. Using the identities from Proposition \ref{prop:derivatives}, we obtain:
\begin{align*}
F''(p) =& \frac{1}{p^2} \int_{\rho \le p}  2\rho  |\nabla A|^2 -\frac{1}{18}\rho |\nabla R|^2 -\frac{1}{3}  A(\nabla R, \nabla \rho)   + \rho A*A*Rm  \\
&+ \oint_{\rho = p} \frac{1}{3} \rho^{-1}  A(\nabla R, \nu) -  C_{n,\gamma}  \frac {R|A|^2}{|\nabla \rho| }, 
\end{align*}
\begin{align*}
G''(p) = \frac{1}{p^2}\int_{\rho \le p}  2\rho  |\nabla R|^2   + \rho A*A*Rm  + \oint_{\rho = p}  \frac{1}{\rho^2}(2\gamma-3)R^2 |\nabla \rho| -  C_{n,\gamma} \frac {R^3}{|\nabla \rho| }. 
\end{align*}
Adding these equations, we derive:
\begin{align*}
(F+cG)''(p) =& \frac{1}{p^2} \int_{\rho \le p}  2\rho  |\nabla A|^2 +(2c-\frac{1}{18}) \rho |\nabla R|^2 -\frac{1}{3}  A(\nabla R, \nabla \rho)  \\
& + \rho(c+1) A*A*Rm  \\
&+ \oint_{\rho = l}\frac{1}{\rho^2}  (2\gamma-3)cR^2 |\nabla \rho| + \frac{1}{3} \rho^{-1}  A(\nabla R, \nu) \\
&-  C_{n,\gamma} (c+1) \frac {R|A|^2}{|\nabla \rho| }. 
\end{align*}

Now, by Lemmas \ref{lemma:algebraicineq} and \ref{lemma:algebraic}, noting that $\rho A(\nabla R, \nabla \rho)$ is a linear combination of $E^2(\nabla \rho,\nabla \rho)$, $RE(\nabla \rho, \nabla \rho)$, and $R^2|\nabla \rho|^2$, we conclude that for every $\epsilon>0$, there exists sufficiently large $c_\epsilon$ and a constant $c_\gamma$(not depending on $\epsilon$) such that
\begin{align*}
(F+c_\epsilon G)''(p) \ge& \frac{1}{p^2} \int_{\rho \le p} - \epsilon c_\epsilon \rho^{-1} R^2 |\nabla \rho|^2 - C_0(3,\gamma) (c_\epsilon+1) C_{n,\gamma} \rho^{-1} |A|^2  \\
&+ \frac{1}{p^2} \oint_{\rho = p}  (2\gamma-3)c_\epsilon R^2 |\nabla \rho| -c_\gamma R^2 |\nabla \rho| - C_0(3,\gamma)  (c_\epsilon+1)C_{n,\gamma} |A|^2 |\nabla \rho|.
\end{align*}
If $C_0(3,\gamma)$ is sufficiently small depending on $\epsilon$, then we obtain the following inequality:
\begin{align*}
(F+c_\epsilon G)''(p) \ge&  -\epsilon \frac{1}{p^2}  (F+c_\epsilon G)(p) -\epsilon\frac{1}{p} (F+c_\epsilon G)'(p) \\
& + \frac{1}{p^2} \oint_{\rho = p}  [(2\gamma-3+\epsilon)c_\epsilon -c_\gamma ]  R^2 |\nabla \rho|.
\end{align*}

\textbf{Case 1:}  $\gamma \ge \frac{3}{2}$.

In this case, we may choose $c_\epsilon$ sufficiently large so that $(2\gamma-3)c_\epsilon - c_\gamma \ge 0$. This leads to the following ordinary differential inequality:
\begin{equation}\label{ineq:ode}
(F+c_\epsilon G)''(p) \ge -\epsilon \frac{1}{p^2}  (F+c_\epsilon G)(p) -\epsilon\frac{1}{p} (F+c_\epsilon G)'(p). 
\end{equation} 
Define $H_\epsilon = F + c_\epsilon G$. The associated indicial equation for the above inequality is:
$$x^2 -(1-\epsilon)x + \epsilon = 0. $$
If $\epsilon$ is chosen sufficiently small, there exists  two indicial roots $0<\alpha <\beta<1$ such that $\beta$ is sufficiently close to 1 and $\alpha$ is sufficiently close to 0.
Rewriting the differential inequality, we obtain:
$$0 \le \big( \frac{ (H(p)/p^\beta)'}{p^{\alpha - \beta -1 }} \big)'.$$

We first prove prove that $\lim_{p \rightarrow 0} \frac{ (H(p)/p^\beta)'}{p^{\alpha - \beta -1 }} = 0$.
This reduces to showing that $\lim_{p \rightarrow 0} \frac{H'(p)}{p^{\alpha-1}} = 0$, which is equivalent to
$$\lim_{p \rightarrow 0} \frac{\oint_{\rho = p } |A|^2 + c_\epsilon R^2} {p^\alpha }=0.$$ 
This follows from the $L^2$-gradient estimate and Hardy's inequality (Proposition \ref{prop:weightedgradient2}), which states that $\oint_{\rho = p} ( |A|^2 + c_\epsilon R^2)/p$ is bounded near $p=0$.
Thus, we conclude that $H(p)/p^\beta$ is non-decreasing in $p$.
Now, by Proposition \ref{prop:curvaturebound}, we obtain: 
$$ H(p) = \int_0^p\frac{1}{q}  \oint_{\rho = q}(|A|^2 + c_\epsilon R^2)|\nabla \rho|\le C + \int_1^p C/ p = C + C \log p .$$
Taking the limit as $p \rightarrow \infty$, we prove that $A$ and $R$ vanish identically on $X$.

\textbf{Case 2:}  $1< \gamma <\frac{3}{2}$.

In this case, the differential inequality takes the form:
$$(F+c_\epsilon G)''(p) \ge -\epsilon \frac{1}{p^2}  (F+c_\epsilon G)(p)  + \frac{1}{p}(2\gamma-3 -\epsilon)  (F+c_\epsilon G)'(p). $$
The corresponding indicial equation is:
$$ x^2 - (2\gamma -2 - \epsilon)x + \epsilon = 0.$$
If $\epsilon$ is chosen sufficiently small depending on $\gamma$, there exists  two indicial roots $0<\alpha <\beta<2\gamma -2 $ such that $\beta$ is sufficiently close to $2\gamma-2$ and $\alpha$ is sufficiently close to 0.

By Proposition \ref{prop:weightedgradient},  $H_\epsilon(p)/p^{2\gamma-2}$ is bounded near $p=0$. Using the same reasoning as in the case $\gamma \ge \frac{3}{2}$, we conclude 
$$\lim_{p \rightarrow 0} \frac{H'(p)}{p^{\alpha-1}} = 0.$$
This identity again ensures the monotonicity $H(p)/p^\beta$. Since $H(p)$grows at most logarithmically, the same argument as above proves that  $R$ and $A$ vanish identically.

Since $A \equiv 0$ in both cases, the main theorem follows from the Hessian rigidity theorem or from Proposition \ref{prop:globgeo}, given that $|\nabla \rho | \equiv 1$ globally.
\end{proof}

\subsection{$n+1 \ge 5$ dimensional case}

In this section, we prove the main theorem for general $n \ge 4$.

\begin{lemma} \label{lemma:algebraic2}
 The coefficient of $E^2(\nabla \rho, \nabla \rho)$ in the expression $$\frac{1}{n} \rho A(\nabla R, \nabla \rho) -2(n-3)  \nabla_m \rho A_{mi} A_{ij} \nabla_j \rho$$ is positive if $\gamma > \tfrac{n}{2}-1$.  In particular, we have
$$ \frac{1}{n} \rho A(\nabla R, \nabla \rho) -2(n-3)  \nabla_m \rho A_{mi} A_{ij} \nabla_j \rho + C_{n,\gamma} R^2|\nabla \rho|^2 \ge 0$$ for some $C_{n,\gamma}>0$.
\end{lemma}

\begin{proof}
By direct computation, the coefficient of $E^2(\nabla \rho, \nabla \rho)$ is 
$$4(\gamma-\tfrac{1}{2}) -2(n-3) = 2(2\gamma -n +2) >0 $$
for $\gamma > \tfrac{n}{2}-1$ ensuring that Lemma \ref{lemma:algebraicineq} is applicable.
\end{proof}

\begin{theorem} 
Let $n \ge 4$ and $\gamma > \frac{n}{2}-1$. There exists a positive number $C_0(n, \gamma)$ with the following significance:

Suppose $(X^{n+1}, \mathbb{R}^n, g_+)$ is a complete, smooth  Poincar\'{e}-Einstein manifold with conformal infinity $(\oR, g_{0})$. Assume there exists the adapted compactification with parameter $\gamma$, $\rho \in C^3(\overline{X})$, and let $g = \rho^2 g_+$ be the compactified metric.
Further, assume the following conditions hold:

(1) $R_g\ge 0$

(2) $|Rm_g|(x) \le \frac{C(n, \gamma)}{\mathrm{dist}(x,o)^2}$ for all $x\in \overline{X}$, where $o$ is a fixed point on the boundary.

Then,  $(\overline{X}, \mathbb{R}^n, \rho^2 g_+)$ is isometric to the standard Euclidean upper half-plane. Equivalently $(X^{n+1}, \mathbb{R}^n,  g_+)$ is isometric to the standard upper-half plane model of hyperbolic space.

\end{theorem}

\begin{proof}
In Proposition \ref{prop:highdimderivatives}, we defined the quantity
$$F_{k,l} (p) = \int_{\rho \le p } \rho^{-1} |A|^{2k} R^{2l}  |\nabla \rho|^2 dV_g, $$
and computed its derivatives.

Let $2N$ be a even integer greater or equal to $\frac{n}{2}$. Our strategy is to derive a differential inequality for a linear combination of $F_{k, l}$. Specifically, We seek a function $G = \sum_{l = 0}^m a_l F_{N-l, l }$ such that a second order ordinary differential inequality, analogous to (\ref{ineq:ode}) holds. The positive constants $a_i$ will be chosen inductively, and in the process, the parameters $\epsilon_k$ in the formula for $F_{k,l}''$ will also be determined inductively.

Throughout this subsection, the indices $k, l$ satisfy $k+l = N$, and we suppress the index $l$, using $F_k$ to denote $F_{k, N-k}$.  The following identities are derived from Proposition \ref{prop:highdimderivatives}:
\begin{align*}
F_{k,l}'(p) = & \oint_{\rho = p } \rho^{-1} |A|^{2k} R^{2l}  |\nabla \rho| dV_g , \\
F_{k}''(p) \ge & \frac{1}{p^2} \int_{\rho \le p} f_{k,0} + f_{k,1} + f_{k,2}+C_{k,l,n,\gamma} \rho A^{2k+2l}*Rm \\
&+ \frac{1}{p^2} \oint_{\rho = p } g_k +(2\gamma-3)(|A|^{2k} R^{2l}) |\nabla \rho| + 2(1-\delta_{0k})(n-s)R^{2l}|A|^{2k} |\nabla \rho| \\
&  -C_{k,l,n,\gamma}\frac{\rho^2}{|\nabla \rho|} A^{2k+2l}*Rm. 
\end{align*}
where we grouped the terms as follows:
\begin{align*}
f_{k,0} =&  k(k-1) \rho R^{2l} |A|^{2k-4} |\nabla |A|^2|^2 +k \rho R^{2l}|A|^{2k-2} |\nabla A|^2  + l(2l-1)\rho |A|^{2k}  R^{2l-2} |\nabla R|^2,   \\
f_{k, 1}  =& -k C_{k, l, n, \gamma}  \big[ |A|^{2k-2} R^{2l}\big(\rho |\nabla R|^2 + \frac{1}{\rho} R^2 |\nabla \rho|^2 \big)  +|A|^{2k-4} R^{2l+2} |\nabla A|^2 \big]\\
& - \frac{klC_{k,l,n,\gamma}} {\epsilon_k }\rho R^{2l} |A|^{2k-2} |\nabla R|^2,  \\
f_{k,2} =& -  klC_{k,l,n,\gamma} \epsilon_k   R^{2l-2} |A|^{2k}\big(\frac{1}{\rho} |A|^2 |\nabla \rho|^2  +\rho |\nabla A|^2 \big),  \\
g_k =&  \frac{k}{n} \rho R^{2l} |A|^{2k-2} A(\nabla R, \nu) -2k(n-3)  R^{2l} |A|^{2k-2} \nabla_m \rho A_{mi} A_{ij} \nu_j. 
\end{align*}
We now explain how the terms are grouped.

Observe that all terms in $f_{k,1}$, except for $\frac{1}{\rho}|A|^{2k-2} R^{2l+2} |\nabla \rho|^2$,  also appear in the expression for $f_{k-1, 0}$ and vanish when $k = 0$. Consequently, these terms can be handled by ensuring that $a_{k-1}$ is sufficiently large relative to $a_k$.  The remaining term, $\frac{1}{\rho}|A|^{2k-2} R^{2l+2} |\nabla \rho|^2$, can be absorbed into a  very small multiple of $F_{k-1}$ in such cases. 

  For the $f_{k,2}$ term, if $\epsilon_k$ is choosen sufficiently small, it can be absorbed into $f_{k+1, 0}$ and a  very small multiple of $F_{k+1}$. This observation provides insight into how the constants $a_k$ and $\epsilon_k$ should be chosen inductively.

\begin{claim}
For every $\epsilon>0$, there exists constants $a_N, a_{N-1}, \cdots, a_0$ such that for $F(p) = \sum_{k=N}^0 a_k F_k ( p ) $, the
following inequality holds for $C_0(n,\gamma)	$ sufficiently small.
\begin{align*}
F''(p) \ge & -\frac{\epsilon}{p^2} F(p) + \frac{1}{p} \big[ \min(n-3, 2\gamma-3) -\epsilon \big] F'(p) \\
& - \frac{1}{p^2} C(n,\gamma, a_N, \cdots , a_0 ) \big[ \int_{\rho \le p } \rho|Rm||A|^{2N} 
+  \oint_{\rho = p } \frac{\rho^2}{|\nabla \rho|}|Rm||A|^{2N} \big]. 
\end{align*}
\end{claim}

\begin{proof}
We choose the constants  $a_N, \epsilon_{N}$ inductively, starting from $a_N, \epsilon_N$ down to $a_0, \epsilon_0$.
 We begin by setting $a_N = \epsilon_N = 1$.

For $1 \le k\le N$, $a_{k-1}$ and $\epsilon_{k-1}$ are chosen to satisfy the following inequalities:
\begin{equation}\label{ineq:inductive}
\begin{cases}
a_k f_{k,1} + \frac{1}{3} a_{k-1}  f_{k-1,0}& \ge -\frac{\epsilon}{6\rho } a_{k-1} R^{2l+2 } |A|^{2k-2} |\nabla \rho|^2 \\
a_k g_k & \ge -\frac{\epsilon}{6}  a_{k-1} R^{2l+2 } |A|^{2k-2} |\nabla \rho| \\
a_{k-1} f_{k-1,2} + \frac{1}{3} a_{k}  f_{k,0} &\ge  -\frac{\epsilon}{6\rho} a_{k} R^{2l } |A|^{2k} |\nabla \rho|^2
\end{cases}
\end{equation}
The first and the second inequalities of the inequality (\ref{ineq:inductive}) is satisfied if $a_{k-1}$ 
  is chosen sufficiently large (for the second inequality, we also use Lemma \ref{lemma:algebraic2}).  Once $a_{k-1}$ is determined, the third inequality  satisfied by choosing  $\epsilon_{k-1}$ sufficiently small. Furthermore, for $k = 1$, $g_0$, $f_{0,2}$ and $f_{0,1}$ ensuring that the induction proceeds smoothly without disruption.

Summing up the inequalities in (\ref{ineq:inductive}) for $k = 0, \cdots ,N$, and noting that
$$\big[(2\gamma-3) + 2(1-\delta_{0k})(n-s)\big] R^{2l}|A|^{2k} |\nabla \rho|\ge \min(n-3, 2\gamma-3)  R^{2l}|A|^{2k} |\nabla \rho|, $$
we obtain the inequality:
\begin{align*}
F''(p) \ge & -\frac{\epsilon}{2p^2} F(p) + \frac{1}{p} \big[ \min(n-3, 2\gamma-3) -\frac{\epsilon}{2} \big] F'(p) \\
&  - \frac{1}{p^2} C(n,\gamma, a_N, \cdots , a_0 ) \big[ \int_{\rho \le p } \rho|Rm||A|^{2N} 
+  \oint_{\rho = p } \frac{\rho^2}{|\nabla \rho|}|Rm||A|^{2N} \big]. 
\end{align*} 
\end{proof}

By choosing $C_0(n,\gamma)$ sufficiently small (depending on the constants $a_k$'s) , we obtain the simplified inequality:
$$F''(p) \ge  -\frac{\epsilon}{p^2} F(p) + \frac{1}{p} \big[ \min(n-3, 2\gamma-3) -\epsilon \big] F'(p).$$
Now, following the ODE argument from the proof of the	 $n+1=4$ case, we get the result.

\end{proof}

\begin{remark}
The range of $\gamma$ can be improved if we assume faster decay of the curvature. In this case, the integer $N$ can be chosen smaller, and while handling the $g_k$ terms, we could achieve a slight improvement by explicitly computing the constant in front of $E^2(\nabla \rho, \nabla \rho)$ in the expression
$$g_k + \min(n-3, 2\gamma-3)  R^{2l}|A|^{2k} |\nabla \rho|$$
and then deriving a similar differential inequality.

\end{remark}


\section{Some examples of conformally Einstein manifold with quadratic curvature decay}

In this section, we provide several explicit examples of non-compact manifold with boundary, which is conformal to some Poincar\'{e}-Einstein manifold in the interior,  and has quadratic curvature decay at infinity. The are coming from stereographic projection of smooth,  non-smooth or singular conformally compact Einstein manifold in the classical sense.

\subsection{Hyperbolic space}
Let us consider the hyperbolic $\mathbb{H}^{n+1}$ in the ball model first:
$$
g_{\mathbb{H}}=\frac{4dz^2}{(1-|z|^2)^2}, \quad z=(z', z_{n+1})\in\mathbb{R}^{n+1},\  |z|<1, 
$$
which takes the round sphere $(\mathbb{S}^n,[g_{\mathbb{S}}])$ as its conformal infinity. Fix  $g_{\mathbb{S}}$ on the boundary and the geodesic normal defining function is  
$$x=\frac{1-|z|}{1+|z|}. $$
For any $f\in C^{\infty}(\mathbb{S}^n)$, there exists a unique solution satisfying
\begin{equation}\label{eq.6.1}
-\Delta_{\mathbb{H}} u-s(n-s)u=0, 	\quad \textrm{in $\mathbb{H}^{n+1}$}
\end{equation}
and  $  u\sim x^{n-s}f $ as $  x\rightarrow 0. $
Moreover $u=P_{\mathbb{H}}(s)f$ where $P_{\mathbb{H}}(s)$ is the Poisson operator. 
Let $e_{n+1}=(0, \cdots, 0, 1)$ be the north pole. 
Then the Poisson kernel  at $e_{n+1}$ is
$$P'_{\mathbb{H}}(s)=C_{n,s}\left(\frac{1-|z|^2}{|z-e_{n+1}|^2}\right)^{s} $$
which is an analytic family of distributions and $C_{n,s}$ is the normalization constant. 
In the upper half plane model of $\mathbb{H}^{n+1}$:
$$
g_{\mathbb{H}}=\frac{d\tilde{x}^2+d\tilde{y}^2}{\tilde{x}^2}, 
\quad \tilde{x}>0,\  \tilde{y}\in\mathbb{R}^n, 
$$
the Poisson kernel at $0$ is
$$
 P''_{\mathbb{H}}(s)=C_{n,s}\frac{\tilde{x}^{s}}{(\tilde{x}^2+|\tilde{y}|^2)^{s}}. 
$$
The Stereographic projection provides an isometry between these two models:
$$
\tilde{x}=\frac{1-|z|^2}{|z-e_{n+1}|^2}, 
\quad
\tilde{y}=\frac{2z'}{|z-e_{n+1}|^2}. 
$$
Notice that
$$
\tilde{x}^2+\tilde{y}^2=|z-e_{n+1}|^2|z+e_{n+1}|^2. 
$$
Using $\rho=C_{n,s}^{-\frac{1}{n-s}}P'_{\mathbb{H}}(n-s)^{\frac{1}{n-s}}$ as conformal factor, we obtain 
$$
g=\rho^2g_{\mathbb{H}}=d\tilde{x}^2+d\tilde{y}^2. 
$$
Notice that $P'_{\mathbb{H}}(n-s)^{\frac{1}{n-s}}$ satisfies the equation 
(\ref{eq.6.1}) with  boundary data equal to the Green' function of the scattering operator $S(s)$ at $e_{n+1}$. Hence it is also an adapted defining function for the boundary $\mathbb{S}^n\backslash\{e_{n+1}\}$ w.r.t. metric $g_{\mathbb{H}}$.

\subsection{Adapted stereographic projection from smooth CCE }
Suppose $(X^{n+1}, g_+)$ is a Poincar\'{e}-Einstein metric with conformal infinity $(M , [\hat{g}])$ of positive Yamabe type. 
Graham-Lee proved in \cite{GL} there are many such examples: given any  $\hat{g}$ which is a small perturbation of the round metric $g_{\mathbb{S}}$ on the sphere, then there exists a Poincar\'{e}-Einstein metric $g_+$ in the ball  taking $(\mathbb{S}^n, [\hat{g}])$ as conformal infinity. 
We consider the adapted stereographic projection for  $(X^{n+1}, g_+)$ as above. The key point is to write the Poisson kernel in local coordinates. 

Fix any $p\in  M$, then we can choose a local coordinates $(x,y)\in \mathbb{R}_+\times \mathbb{R}^n$ such that $x$ is the geodesic normal defining function w.r.t. $(g_+, \hat{g})$, and $y$ is the boundary  normal coordinates w.r.t. $\hat{g}$. This implies
\begin{equation*}
\label{metric.1}
g_+=x^{-2}\left(dx^2 + \hat{g}+x^2g_2+\cdots \right), \quad \hat{g}= dy^2 +O(y^2). 
\end{equation*}
So we can write $g_+=x^{-2}\left(dx^2 + \hat{g}+O(r^2)(dx,dy)\right )$
where $O(r^2)(dx,dy)$ denotes the symmetric two tensors in basis $\{dx,dy\}$ with coefficients $O(r^2)$. The regularity of these coefficients follows from the regularity assumption of $x^2g_+$. 

Since $Y(\partial M, [\hat{g}])>0$, $\mathrm{Spec}(-\Delta_+)\geq \frac{n^2}{4}$. For any $s\in \mathbb{C}$ with $\mathrm{Re}(s)>n/2$, the Poisson kernel $P(s)$ is well defined in the following sense: given any $f\in C^{\infty}(M)$, then $u=P(s)f$ is the unique solution satisfying the following equation: 
\begin{equation}\label{eq.6.2}
-\Delta_+ u-s(n-s)u=0, \quad  \textrm{in $X$}	
\end{equation}
and $u\sim x^{n-s}f $ as $ x\rightarrow 0$. In \cite{GZ}, it shows that the  $P(s)$ can be represented by the resolvent $R(s)=(-\Delta_+ u-s(n-s))^{-1}$. 


Fix $s>\frac{n}{2}$ and consider a Poinca\'{e}-Einstein metric    $g_+$   which is smoothly conformally compact and sufficiently close to the hyperbolic 
space in the sense 
$$
\| x^2g_+ -x^2g_{\mathbb{H}}\|_{C^{k,\alpha}(\overline{\mathbb{B}^{n+1}})}
\leq \epsilon. 
$$
for some $k+\alpha>\max\{4,2s-n\}$ and $\epsilon>0$ sufficiently small. 
Due to the work of \cite{MM}, \cite{Gu} and \cite{GQ},  in this case the Poisson kernel is analytic in a neighborhood of $[n-s, s]$.
By \cite{JS, GZ}, letting $r=\sqrt{x^2+y^2}\in [0,\delta_0)$ near $p=e_{n+1}$, then $P(s)$ at $p$ is
$$P(s)= C_{n,s} x^{s}r^{-2s}F$$
where 
$F$ is smooth in the blow-up space of $\overline{M}\times\partial M$. (See \cite{JS} for exact meaning of the blow-up.) Here we write it in local coordinates in the following way: 
$$
F=\begin{cases} 
\displaystyle F_1\left(r, \frac{x}{r},  \frac{y}{|y|}\right), 
\quad \textrm{for}\ x<\sqrt{3}|y|, 
\\
\displaystyle  F_2\left(r, \frac{y}{r}\right), 
\quad \textrm{for}\  |y|<\sqrt{3}x, 
\end{cases}
$$
for some 
 $F_1\in C^{\infty} ([0,\delta_0)\times [0,1/2) \times \mathbb{S}^{n-1}$
and $F_2\in C^{\infty}([0,\delta_0) \times \mathbb{B}^{n}(2).$ Moreover, 
$F_1(0,\cdot, \cdot )=1, F_2(0,\cdot)=1$. Since in  $X$, 
$$-\Delta_+P(n-s)-s(n-s)P(n-s)=0$$
using the asymptotical expansion of $g_+$ and $\Delta_+$ near $p$, we can obtain that 
$$F=1+O(r^2). $$
Define $\rho= C_{n,n-s}^{-\frac{1}{n-s}} P(n-s)^{\frac{1}{n-s}}$which is smooth in $\overline{X}\backslash\{p\}$. Then  $(\overline{X}\backslash\{p\}, g=\rho^2g_+)$ is a manifold with noncompact boundary $M\backslash\{p\}$ and  near $p$, 
$$
g=r^{-4}(1+O(r^2))\left(dx^2+dy^2+O(r^2)(dx,dy)\right). 
$$
Here $O(r^2)(dx,dy)$ means symmetric $2$-tensor with coefficients $O(r^2)$. 
Let 
$$
\tilde{x}=\frac{x}{x^2+y^2}, \quad \tilde{y}=\frac{y}{x^2+y^2},\quad  \tilde{r}=\sqrt{\tilde{x}^2+\tilde{y}^2}=\frac{1}{r}.
$$
Then for $\tilde{r}>\delta_0^{-1}$,  
$$
g=d\tilde{x}^2+d\tilde{y^2}+O(\tilde{r}^{-2}) (d\tilde{x}, d\tilde{y}). 
$$
Hence in $(\tilde{x}, \tilde{y})$ coordinates, 
$$
R_{ijkl}=O(\tilde{r}^{-4}). 
$$
Notice that  there exist constant $C>0$ such that 
$$
\frac{1}{C} \leq\frac{\mathrm{dist}_{\bar{g}}(\cdot, -e_{n+1})}{\tilde{r}}  \leq C. 
$$

\subsection{Stereographic projection from non-smooth CCE }
If the conformal compactification of $(X^{n+1}, g_+)$ has lower regularity or singular, then the existence of Poisson kernel is a problem, and hence we can not define an adapted defining function. 

 In \cite{BL}, Bahuaud-Lee constructed some Poincr\'{e}-Einstein metric $g_+$ in the ball $ \mathbb{B}^{n+1}$ which has only  $C^{1,1,}$ conformal compactification and its conformal infinity is a $C^{1,1}$ metric $\hat{g}$ on the sphere $ \mathbb{S}^n$. Fix   $p\in  \mathbb{S}^n$, $g_+$  has the following expansion: 
$$
g_+=x^{-2}\left(dx^2 + \hat{g}+O(r^2)(dx,dy) \right), \quad r=\sqrt{x^2+y^2}<\delta_0. 
$$
Here the $O(r^2)$ coefficients has $C^{1,1}$ regularity in $x,y$ coordinates. 

Let $\rho=\frac{x}{x^2+y^2}$ for $r<\frac{1}{2}\delta_0$. Then $\rho$ is smooth in $\overline{\mathbb{B}^{n+1}}\backslash\{p\}$ and $\rho$ is a regular boundary defining function for $r\geq \delta_0$. Define $g=\rho^2g_+$. 
For $r<\delta_0$, let 
$$ 
\tilde{x}=\frac{x}{x^2+y^2}, \quad  \tilde{y}=\frac{y}{x^2+y^2}, \quad  \tilde{r}=\sqrt{\tilde{x}^2+\tilde{y}^2}. $$
Then for $\tilde{r}>\delta_0$, we also have 
$$
g=d\tilde{x}^2+d\tilde{y^2}+O(\tilde{r}^{-2}) (d\tilde{x}, d\tilde{y}), 
$$
The $C^{1,1}$ regularity of $O(r^2)(dx,dy)$ implies that in $(\tilde{x}, \tilde{y})$ coordinates, 
$$
R_{ijkl}=O(\tilde{r}^{-4}), \quad as\  \tilde{r}\rightarrow \infty.
$$
Notice here $g$ has $C^{1,1}$ regularity  up to boundary too.

\subsection{Generalized stereographic projection from singular CCE}
If the conformal compactification of $(X^{n+1}, g_+)$ has a singular set, which is a submanifold on the boundary, then we can blow up the submanifold and obtain the generalized stereographic projection metric. 

  In \cite{AOS},  Alvarado-Ozuch-Santiago studied some smooth  Poincar\'{e}-Einstein metric  $g_+$ in the interior of $X^4$, whose conformal infinity $(M^3,[\hat{g}])$ has conic-edge singularity. More explicitly, 
$$
\begin{aligned}
g_+ =\frac{1}{(x-y)^2} &\left[ 
\frac{Q(y)}{1-x^2y^2}(d\psi-x^2 d\varphi)^2
+\frac{1-x^2y^2}{Q(y)}dy^2\right.
\\
& \left. 
\frac{P(x)}{1-x^2y^2}(d\varphi-y^2d\psi)^2+\frac{1-x^2y^2}{P(x)}dx^2
\right]
\end{aligned}
$$
where $\psi, \varphi$ are periodic type coordinates and 
$(x,y)\in \Omega\subset \mathbb{R}^2$ such that $P(x)>0, Q(y)>0$. 
Here 
$$
\begin{aligned}
P(x)&=b(x-\alpha_1)[(x-1+\alpha_2)^2+\alpha_3](x-\alpha_4)
\\
&=bx^4+cx^3+dx^2+ex+b+1	, 
\end{aligned}
$$
with $b=(-1+\alpha_1\alpha_2^2\alpha_4-\alpha_4-2\alpha_1\alpha_2\alpha_4+\alpha_1\alpha_3\alpha_4)^{-1}$ and 
$$
-Q(y)=P(y)+y^4-1=(b+1)y^4+cy^3+dy^2+ey+b. 
$$
By choosing the constant
$$
-1<\alpha_1<0, \quad \alpha_2,\alpha_3<0 \ \textrm{small}, \quad \alpha_2^2+\alpha_3=0, \quad \alpha_4>1, 
$$
we have
$$
\Omega=\{\alpha_1\leq x<y\leq 1\}. 
$$
By choosing the periods of $\psi, \varphi$ properly, we can make $g_+$ smooth  when $x\rightarrow \alpha_1$or $y\rightarrow 1$. And hence $(X^4, g_+)$ is smooth in the interior. Here  $\{x=y\}$ corresponding to the conformal infinity and $\rho=y-x$ is a boundary defining function. The boundary metric is
$$
\begin{aligned}
\hat{g}=\rho^2g_+|_{\alpha=\beta}= &  \ 
(1-x^4)\left(\frac{1}{P(x)}+\frac{1}{Q(x)}\right)dx^2+ 
\frac{Q(x)}{1-x^4}(d\psi-x^2 d\varphi)^2
	\\
	& +\frac{P(x)}{1-x^4}(d\varphi-x^2 d\psi)^2. 
\end{aligned}
$$
When $x=y\rightarrow 1$, by changing boundary coordinates, we have
$$
\hat{g}=dt^2+\alpha^2t^2d\theta_1^2+\beta^2d\theta_2, 
\quad \textrm{on $(t,\theta_1,\theta_2)\in [0,\delta )\times [0,2\pi ]\times  [0, 2\pi ]$}, 
$$
for some $\alpha>0,\beta>0$. Hence $\{t=0\}$ is the singular set $\Sigma=\{(x,y)=(1,1)\} $ on the boundary. See \cite[Theorem B and Section 4.3]{AOS}.

By directly calculation, the pointwise norm of the Riemannian tensor of $g_+$ is
$$
|Rm_{g_+}|_{g_+}^2 = 24+24(x-y)^6\left(
\frac{k_+^2}{(1+xy)^6}+\frac{k_-^2}{(1-xy)^6}
\right), 
$$
where $k_{\pm}$ are determined by $c=k_++k_-, e=k_+-k_-$. Moreover
 $$|W_{g_+}|_{g_+}=O\left(\frac{(y-x)^3}{(1-x^2y^2)^3}\right ). $$ 
Denote $\bar{g}=(y-x)^2g_+$. Then $\bar{g}$ is smooth everywhere except $(x,y)=(1,1)$. When $(x,y)\rightarrow (1,1)$, we have
$$|Rm_{\bar{g}}|_{\bar{g}}=O\left(\frac{1}{(1-x^2y^2)^2}\right ), \quad 
|W_{\bar{g}}|_{\bar{g}}=O\left(\frac{(y-x) }{(1-x^2y^2)^3}\right ). 
$$

Now considering a blow-up of the singular set  $\Sigma $. Let $\bar{x}=1-x, \bar{y}=1-y$. Then a collar neighborhood of $\Sigma $ in $(\overline{X}^4, \bar{g})$ is give by 
$$U=\{0\leq \bar{y}\leq \bar{x}<\delta \}.$$ 
Then in $U$, $1-x^2y^2\sim 2(\bar{x}+\bar{y})$, which is also equivalent to $\sqrt{\bar{x}^2+\bar{y}^2}$, $\bar{x}$ and $\mathrm{dist}_{\bar{g}}(\cdot, \Sigma)$. Define
$\tilde{g}=(1-x^2y^2)^{-4} \bar{g}$ and let $\tilde{r}=\mathrm{dist}_{\tilde{g}}(\cdot, p)$ for some fixed interior point $p$. Then $  (1-x^2y^2)=O(\tilde{r}^{-1})$. Moreover, 
$$|Rm_{\tilde{g}}|_{\tilde{g}}=O\left((1-x^2y^2)^2\right )=O(\tilde{r}^{-2}), 
$$
$$
|W_{\tilde{g}}|_{\tilde{g}}=O\left( (y-x)  (1-x^2y^2) \right )=O(\rho \tilde{r}^{-2}), 
$$
where $\rho=\frac{y-x}{1-x^2y^2}$ extends to a bounded boundary defining function for $(\overline{X}\backslash \Sigma, \tilde{g})$.

\appendix
\section{Proof of Lemma \ref{lemma:eqcurvature}}

\begin{lemma} 
For the adapted boundary compactification of Poincar\'{e}-Einstein manifold with parameter $s=\frac{n}{2}+\gamma$, we have:
$$
\begin{aligned}
	 \Delta A_{ij}
= & J_{ij}-\tfrac{n-3}{2}\rho^{-1}\rho^m\left(A_{mi,j}+A_{mj,i}\right)+(n-3)\rho^{-1}\langle \nabla \rho, \nabla A_{ij}\rangle
\\
& -2R_{imjk}A^{mk} 
+(n-3)(A^2)_{ij}
+|A|^2g_{ij}  
+2JA_{ij}.
\end{aligned}
$$
\end{lemma}

\begin{proof}
Notice that the Schouten tensor of $g$ and its trace are
$$
	A_{ij}=\tfrac{1}{n-1}\left(R_{ij}-Jg_{ij}\right), \quad J=\tfrac{1}{2n} R=\mathrm{tr}_gA. 
$$
Therefore, 
$$
	A_{ij}=-\rho^{-1}\rho_{ij}-\tfrac{1}{2s-n-1} J g_{ij}. 
$$
To calculate $\Delta A_{ij}$:
$$
\Delta A_{ij}=-\Delta (\rho^{-1}\rho_{ij})-\tfrac{1}{2s-n-1}( \Delta J) g_{ij} =I+II \cdot  g_{ij} . 
$$
Here by (1), 
$$
	\Delta J=-(n+3-2s)\rho^{-1}\langle \nabla \rho,\nabla J\rangle  
	-(2s-n-1)|A|^2+ \tfrac{n+1}{2s-n-1}J^2. 
$$
Hence 
$$
II=\tfrac{n+3-2s}{2s-n-1}\rho^{-1}\langle \nabla \rho,\nabla J\rangle  
+|A|^2-\tfrac{n+1}{(2s-n-1)^2}J^2 . 
$$
For $\Delta (\rho^{-1}\rho_{ij})$, direct computation shows that
$$
\begin{aligned}
(\rho^{-1}\rho_{ij})_{,k} 
& =\rho^{-1}\rho_{ijk}-\rho^{-2}\rho_{ij}\rho_k, 	
\\
(\rho^{-1}\rho_{ij})_{,k}^{\ \ k}
&= \rho^{-1}\rho_{ijk}^{\ \ \ k} -2\rho^{-2}\rho_{ijk}\rho^k
+2\rho^{-3} \rho_{ij}|\nabla \rho|^2-\rho^{-2}(\Delta \rho) \rho_{ij}. 
\end{aligned}
$$
Because
$$
\begin{aligned}
A_{ij,k} 
& =-\rho^{-1}\rho_{ijk}+\rho^{-2}\rho_{ij}\rho_k
-\tfrac{1}{2s-n-1} J_k g_{ij}, 
\\
\rho^{-1}\langle \nabla \rho,\nabla A_{ij}\rangle 
& =-\rho^{-2}\rho_{ijk}\rho^k+\rho^{-3}\rho_{ij}|\nabla \rho|^2
-\tfrac{1}{2s-n-1} \rho^{-1}\langle \nabla \rho, \nabla J\rangle g_{ij}	, 
\\
\Delta\rho
& =-\tfrac{2s}{2s-n-1} \rho J. 
\end{aligned}
$$
We have
$$
\begin{aligned}
\Delta (\rho^{-1}\rho_{ij})   & =  \rho^{-1}\rho_{ijk}^{\ \ \ k}
+2\left(\rho^{-1}\langle \nabla \rho,\nabla A_{ij}\rangle
+\tfrac{1}{2s-n-1} \rho^{-1}\langle \nabla \rho, \nabla J\rangle g_{ij}\right)	
\\
&\quad 
+ \tfrac{2s}{2s-n-1} \rho^{-1}\rho_{ij}J 
\\
 &= \rho^{-1}\rho_{ijk}^{\ \ \ k}
+2\left(\rho^{-1}\langle \nabla \rho,\nabla A_{ij}\rangle
+\tfrac{1}{2s-n-1} \rho^{-1}\langle \nabla \rho, \nabla J\rangle g_{ij}\right)	\\
&\quad 
 -\tfrac{2s}{2s-n-1}J A_{ij} -\tfrac{2s}{(2s-n-1)^2}J^2 g_{ij}. 
\end{aligned}
$$
So we obtain
$$
\begin{aligned}
 I=& -\rho^{-1}\rho_{ijk}^{\ \ \ k}
-2\left(\rho^{-1}\langle \nabla \rho,\nabla A_{ij}\rangle
+\tfrac{1}{2s-n-1} \rho^{-1}\langle \nabla \rho, \nabla J\rangle g_{ij}\right)	\\
& +\tfrac{2s}{2s-n-1}J A_{ij} +\tfrac{2s}{(2s-n-1)^2}J^2 g_{ij}, 
\end{aligned}
$$
and
$$
\begin{aligned}
\Delta A_{ij}=& - \rho^{-1}\rho_{ijk}^{\ \ \ k} 
-2 \rho^{-1}\langle \nabla \rho,\nabla A_{ij}\rangle
- \rho^{-1}\langle \nabla \rho, \nabla J\rangle g_{ij}
\\
& +\frac{2s}{2s-n-1}J A_{ij}+|A|^2g_{ij}+\frac{1}{2s-n-1}J^2 g_{ij}
\\
=& - \rho^{-1}\rho_{ijk}^{\ \ \ k}+III_1 +III_2. 
\end{aligned}
$$
Here $III_1$ denotes the linear curvature term and $III_2$ denotes the quadratic curvature term. 
Now, using the convention $g_{im}R^m_{\ jkl}=R_{ijkl}$, we have
$$
\begin{aligned}
\rho_{ijkl}=&\ \left(\rho_{ikj}-R_{imjk}\rho^m\right)_{,l}	
\\
=&\ \rho_{kijl}    - R_{imjk}\rho^m_{\ \ l}   - R_{imjk,l}\rho^m
\\
=&\  \rho_{kilj} - R_{kmjl}\rho_i^{\  m}-R_{imjl}\rho_k^{\ m}  
- R_{imjk}\rho^m_{\ \ l}   - R_{imjk,l}\rho^m
\\
=&\ \left(\rho_{kli}-R_{kmil}\rho^m
\right)_{,j}
\\
&\ 
- R_{kmjl}\rho_i^{\  m}-R_{imjl}\rho_k^{\ m}  
- R_{imjk}\rho^m_{\ \ l}   - R_{imjk,l}\rho^m
\\
=&\ \rho_{klij} -R_{kmil}\rho^m_{\ j} -R_{kmil,j}\rho^m
\\
&\ 
- R_{kmjl}\rho_i^{\  m}-R_{imjl}\rho_k^{\ m}  
- R_{imjk}\rho^m_{\ \ l}   - R_{imjk,l}\rho^m, 
\end{aligned}
$$
and
$$
\begin{aligned}
\rho^{-1}\rho_{ijk}^{\ \ \ k}
&= \rho^{-1}\left( \rho_{k\ ij}^{\ k}
+R_{mi}\rho_{j}^{\ m} +R_{mj}\rho_{i}^{\ m}
-2R_{imjk}\rho^{mk} \right.
\\
&\quad 
\left.+R_{mi,j}\rho^m-R_{imjk,}^{\ \ \ \ \ \ k}\rho^m
\right)
=A+B+C+D.	
\end{aligned}
$$
Here
$$
\begin{aligned}
	A=&\ \rho^{-1} \rho_{k\ ij}^{\ k}= \rho^{-1}(\Delta\rho)_{,ij}
=-\tfrac{2s}{2s-n-1} \rho^{-1}(\rho J)_{,ij}
\\
=&\ -\tfrac{2s}{2s-n-1}\left[J_{ij}
+\rho^{-1} \left(\rho_i J_j+\rho_j J_i \right)
+\rho^{-1}\rho_{ij}J
\right]
\\
=&\  -\tfrac{2s}{2s-n-1}J_{ij}
-\tfrac{2s}{2s-n-1}\rho^{-1} \left(\rho_i J_j+\rho_j J_i \right)
\\
&\ 
+\tfrac{2s}{2s-n-1}JA_{ij}+\tfrac{2s}{(2s-n-1)^2}J^2 g_{ij}
\\
=& A_1+A_2, 
\end{aligned}
$$
where  $A_1$ denotes the linear curvature term and $A_2$ denotes the quadratic curvature term; and 
$$
\begin{aligned}
	B=&\ \rho^{-1} \left(R_{mi}\rho_{j}^{\ m} +R_{mj}\rho_{i}^{\ m}\right)
	\\
=&\ -2(n-1)(A^2)_{ij}-\tfrac{4s-4}{2s-n-1}J A_{ij}
-\tfrac{2}{2s-n-1}J^2g_{ij}
=B_2,	
\\
C=&\ -2\rho^{-1}R_{imjk}\rho^{mk}
= 2R_{imjk}A^{mk}+\tfrac{2}{2s-n-1}JR_{ij}	
\\
=&\  2R_{imjk}A^{mk}+\tfrac{2(n-1)}{2s-n-1}JA_{ij}	 +\tfrac{2}{2s-n-1}J^2g_{ij}=C_2,
\end{aligned}
$$
where $B_2,C_2$ means $B, C$ are both quadratic curvature term. 
By the second Bianchi identity, we have
$$
R_{jkim,l}+R_{jkli,m}+R_{jkml,i}=0, \quad \Rightarrow \quad
R_{jkim,}^{\ \ \ \ \ \ k}-R_{ji,m}+R_{jm,i}=0. 
$$
Hence
$$
R_{mi,j}-R_{imjk,}^{\ \ \ \ \ \  k}=R_{mi,j}+R_{mj,i}-R_{ij,m}, 
$$
and
$$
\begin{aligned}
D=&\ \rho^{-1}\rho^m\left(R_{mi,j}-R_{imjk,}^{\ \ \ \ \ \  k}\right) \\
=& 
(n-1)\rho^{-1}\rho^m\left(A_{mi,j}+A_{mj,i}\right)-(n-1)\rho^{-1}\langle \nabla \rho, \nabla A_{ij}\rangle	
\\
& +\rho^{-1}\left(\rho_i J_j +\rho_j J_i\right)-\rho^{-1}\langle \nabla \rho, \nabla J\rangle g_{ij}=D_1,
\end{aligned}
$$
where $D_1$ means $D$ is only linear curvature term. 
Notice that
$$
\begin{aligned}
A_1+D_1=& 
-\tfrac{2s}{2s-n-1}J_{ij}
-\tfrac{2s}{2s-n-1}\rho^{-1} \left(\rho_i J_j+\rho_j J_i \right)
\\
& + (n-1)\rho^{-1}\rho^m\left(A_{mi,j}+A_{mj,i}\right)-(n-1)\rho^{-1}\langle \nabla \rho, \nabla A_{ij}\rangle	
\\
& +\rho^{-1}\left(\rho_i J_j +\rho_j J_i\right)-\rho^{-1}\langle \nabla \rho, \nabla J\rangle g_{ij} 
\\
=&  -\tfrac{2s}{2s-n-1}J_{ij}   -\tfrac{n+1}{2s-n-1}\rho^{-1} \left(\rho_i J_j+\rho_j J_i \right)-\rho^{-1}\langle \nabla \rho, \nabla J\rangle g_{ij}
\\
&+ (n-1)\rho^{-1}\rho^m\left(A_{mi,j}+A_{mj,i}\right)-(n-1)\rho^{-1}\langle \nabla \rho, \nabla A_{ij}\rangle .	
\end{aligned}
$$
Therefore, the linear curvature term in $-\Delta A_{ij}$ is given by
$$
\begin{aligned}
& \ A_1+D_1-III_1
\\
=&\  -\tfrac{2s}{2s-n-1}J_{ij}
-\tfrac{n+1}{2s-n-1}\rho^{-1} \left(\rho_i J_j+\rho_j J_i \right)
\\
&\ 
+ (n-1)\rho^{-1}\rho^m\left(A_{mi,j}+A_{mj,i}\right)
-(n-3)\rho^{-1}\langle \nabla \rho, \nabla A_{ij}\rangle
\\
=&\  -J_{ij}-\tfrac{n+1}{2s-n-1}\left[J_{ij}+\rho^{-1} \left(\rho_i J_j+\rho_j J_i \right)\right]
\\
&\  + (n-1)\rho^{-1}\rho^m\left(A_{mi,j}+A_{mj,i}\right)
-(n-3)\rho^{-1}\langle \nabla \rho, \nabla A_{ij}\rangle , 
\end{aligned}
$$
and the quadratic form in $-\Delta A_{ij}$ is given by
$$
\begin{aligned}
&\ A_2+B_2+C_2-III_2
\\
=&\  2R_{imjk}A^{mk}-2(n-1)(A^2)_{ij}-|A|^2g_{ij}  -2JA_{ij}
+\tfrac{n+1}{(2s-n-1)^2}J^2g_{ij} . 
\end{aligned}
$$
At last, we can simplify the linear term by the following identities
$$
|\nabla\rho|^2_{ij}=|\nabla \rho|^2_{ji} \quad \Rightarrow \quad \rho_{kij}\rho^k=\rho_{kji}\rho^k. 
$$
Hence
$$
\begin{aligned}
J_i=&\ \tfrac{2s-n-1}{2}\left[-2\rho^{-2}\rho_{ki}\rho^k-2\rho^{-3}(1-|\nabla\rho|^2)\rho_i\right], 
\\
J_{ij}=& \tfrac{2s-n-1}{2}\left[-2\rho^{-2}\rho_{kij}\rho^k-2\rho^{-2}\rho_{ki}\rho^{k}_{\ j} -2\rho^{-3}(1-|\nabla\rho|^2)\rho_{ij}\right.
\\
&\quad\quad\quad \quad\quad + \left. 4\rho^{-3}\left(\rho_{ki}\rho^k\rho_j+\rho_{kj}\rho^k\rho_i\right)+6\rho^{-4}(1-|\nabla\rho|^2)\rho_i\rho_j
\right], 
\end{aligned}
$$
and
$$
\begin{aligned}
&\ \rho^{-1}\left(\rho_i J_j+\rho_j J_i \right)\\
=&\  \tfrac{2s-n-1}{2}\left[-2\rho^{-3}\left(\rho_{ki}\rho^k\rho_j+\rho_{kj}\rho^k\rho_i\right)-4\rho^{-4}(1-|\nabla\rho|^2)\rho_i\rho_j\right]
\\
=&\  -(2s-n-1)\left[ \rho^{-3}\left(\rho_{ki}\rho^k\rho_j+\rho_{kj}\rho^k\rho_i\right)+2\rho^{-4}(1-|\nabla\rho|^2)\rho_i\rho_j\right], 
\\
&\ \rho^{-1}\rho^k\left(A_{ki,j}+A_{kj,i}\right)\\
= &\ 
-2\rho^{-2}\rho_{kij}\rho^k+\rho^{-3}\left(\rho_{ki}\rho^k\rho_j+\rho_{kj}\rho^k\rho_i\right)
-\tfrac{1}{2s-n-1}\rho^{-1}\left(\rho_i J_j+\rho_j J_i \right)
\\
=&\  -2\rho^{-2}\rho_{kij}\rho^k 
+2\rho^{-3}\left(\rho_{ki}\rho^k\rho_j+\rho_{kj}\rho^k\rho_i\right)
+2\rho^{-4}(1-|\nabla\rho|^2)\rho_i\rho_j. 
\end{aligned}
$$
These together implies that
$$
\begin{aligned}
&\ J_{ij}+\rho^{-1}\left(\rho_i J_j+\rho_j J_i \right)\\
=&\  \tfrac{2s-n-1}{2}\left[-2\rho^{-2}\rho_{kij}\rho^k-2\rho^{-2}\rho_{ki}\rho^{k}_{\ j} -2\rho^{-3}(1-|\nabla\rho|^2)\rho_{ij}\right.
\\
&\  
+ \left. 2\rho^{-3}\left(\rho_{ki}\rho^k\rho_j+\rho_{kj}\rho^k\rho_i\right)+2\rho^{-4}(1-|\nabla\rho|^2)\rho_i\rho_j
\right]
\\
=&\ \tfrac{2s-n-1}{2}\rho^{-1}\rho^k\left(A_{ki,j}+A_{kj,i}\right)
\\
&\  -(2s-n-1)\left(\rho^{-2}\rho_{ki}\rho^{k}_{\ j}+\rho^{-3}(1-|\nabla\rho|^2)\rho_{ij}  \right). 
\end{aligned}
$$
Therefore
$$
\begin{aligned}
 &\ -\tfrac{n+1}{2s-n-1}\left[J_{ij}+\rho^{-1} \left(\rho_i J_j+\rho_j J_i \right)\right]
+ (n-1)\rho^{-1}\rho^k\left(A_{ki,j}+A_{kj,i}\right)
\\
= & -\tfrac{n+1}{2}\rho^{-1}\rho^k\left(A_{ki,j}+A_{kj,i}\right) +(n+1)\left(\rho^{-2}\rho_{ki}\rho^{k}_{\ j}+\rho^{-3}(1-|\nabla\rho|^2)\rho_{ij}  \right)
\\
&\ 
+ (n-1)\rho^{-1}\rho^k\left(A_{ki,j}+A_{kj,i}\right)
\\
= &\ 
\tfrac{n-3}{2}\rho^{-1}\rho^k\left(A_{ki,j}+A_{kj,i}\right)
+(n+1)\left(\rho^{-2}\rho_{ki}\rho^{k}_{\ j}+\rho^{-3}(1-|\nabla\rho|^2)\rho_{ij}  \right).
\end{aligned}
$$
The new quadratic term is given by
$$
\begin{aligned}
 \rho^{-2}\rho_{ki}\rho^{k}_{\ j}=&\ (A^2)_{ij}+\frac{2}{2s-n-1}JA_{ij}+\frac{1}{(2s-n-1)^2}J^2g_{ij}, 	
 \\
 \rho^{-3}(1-|\nabla\rho|^2)\rho_{ij}=&\ -\frac{2}{2s-n-1}JA_{ij}-\frac{2}{(2s-n-1)^2}J^2g_{ij}. 
\end{aligned}
$$
Hence
$$
(n+1)\left(\rho^{-2}\rho_{ki}\rho^{k}_{\ j}+\rho^{-3}(1-|\nabla\rho|^2)\rho_{ij}  \right)
=(n+1)(A^2)_{ij}- \tfrac{n+1}{(2s-n-1)^2}J^2g_{ij}. 
$$
We finish the proof. 
\end{proof}



\begin{thebibliography}{99}

\bibitem[An1]{An1} M. Anderson,
{Boundary regularity, uniqueness, and non-uniqueness for AH Einstein metrics on 4-manifolds}, Adv. Math. 179 (2003), no. 2, 205-249.



\bibitem[An2]{An2} M. Anderson,
{Einstein metrics with prescribed conformal infinity on 4-manifolds}, Geom. Funct. Anal. 18 (2008), no. 2, 305-366. 

\bibitem[An3]{An3} M. Anderson,
 {Geometric aspects of the AdS/CFT Correspondence},  AdS/CFT correspondence: Einstein metrics and their conformal boundaries, 1-31,  IRMA Lect. Math. Theor. Phys., 8, Eur. Math. Soc., Zurich, 2005. 


\bibitem[An4]{An4}
M. Anderson, 
{$L^2$ curvature and volume renormalization for AHE metrics on 4 manifolds}, 
Math. Res. Lett. 8 (2001), 171-188.



\bibitem[AOS]{AOS}
C. A. Alvarado, T. Ozuch, D. A. Santiago,
{Families of degenerating Poincare-Einstein metrics on $\mathbb{R}^4$}, 
Annals of Global Analysis and Geometry (2024) 65:5


\bibitem[AV]{AV} A. Ache, J. Viaclovsky, {
Obstruction-flat asymptotically locally Euclidean metrics},
 Geom. Funct. Anal. 22 (4) (2012) 832–877.





\bibitem[BL]{BL}
E. Bahuaud, J. Lee, 
{Low regularity Poincar\'{e}-Einstein metrics}.
Proc. Amer. Math. Soc.146 (2018), no.5, 2239-2252.


\bibitem[BKN]{BKN} S. Bando, A. Kasue, H. Nakajima,
{On a construction of coordinates at infinity on manifolds with fast curvature decay and maximal volume growth},
 Invent. Math., 97(2) (1989), 313–350.

\bibitem[C]{C} T. H. Colding. New monotonicity formulas for Ricci curvature and applications. I. Acta Mathematica, 209(2):229–263, 2012.



\bibitem[CC]{CC}
J. Case J, S.Y.A. Chang.
{On the fractional GJMS operators},
Comm. Pure Appl. Math, 2016, {69}(6), 1017-1061.


\bibitem[CDLS]{CDLS}
P. Chru\'{s}ciel , E. Delay, J. Lee, D. Skinner,
{Boundary regularity of conformal compact einstein metrics}.
J. Differential Geom, 2005, {69}, 111-136.


\bibitem[CGJQ]{CGJQ}
S.Y.A. Chang, Y. Ge, X. Jin, J. Qing,
{Perturbation compactness and uniqueness for a class of conformally compact Einstein manifolds.}
Adv. Nonlinear Stud. 24(2024), no.1, 247-278.

\bibitem[CG1]{CG1}
S.Y.A. Chang, Y. Ge,
{Compactness of conformally compact Einstein manifolds in dimension 4.}
Adv. Math. 340 (2018), 588-652. 

\bibitem[CG2]{CG2}
S.Y.A. Chang, Y. Ge,
{On a problem of conformal fill in by Poincare Einstein metric}. 
Preprint, 2025. 


\bibitem[CGQ]{CGQ}
S.Y.A. Chang, Y. Ge, J. Qing, 
{Compactness of conformally compact Einstein 4-manifolds II},  
Adv. Math. 373 (2020), 107325, 33 pp. 

\bibitem[CLW]{CLW1}
X. Chen, M. Lai, F. Wang, 
{Escobar-Yamabe compactifications for Poincar\'{e}-Einstein manifolds and rigidity theorems}, 
Adv. Math. 343 (2019), 16-35. 

\bibitem[CM1]{CM1} T.H. Colding, T. H.,  W.P. Minicozzi, 
{Harmonic functions with polynomial growth},
 J.  Diff. Geom., 46(1) (1997), 1-77.


\bibitem[CM2]{CM2} T.H.  Colding, W.P. Minicozzi,
 {Harmonic functions on manifolds},
 Annals of mathematics 146, no. 3 (1997), 725-747.


\bibitem[CM3]{CM3} T.H.  Colding, W.P. Minicozzi,
{Ricci curvature and monotonicity for harmonic functions},
 Calculus of Variations and Partial Differential Equations, 49(3) (2014), 1045–1059.

\bibitem[CM4]{CM4}
 T.H.  Colding, W.P. Minicozzi, 
{On uniqueness of tangent cones for Einstein manifolds},
 Invent. math., 196 (2014) 515–588.



\bibitem[CQY]{CQY}
S.Y.A Chang, J. Qing, P. Yang,
On the topology of conformally compact Einstein 4-manifolds.
Noncompact problems at the intersection of geometry, analysis, and topology, 49-61.
Contemp. Math., 350
American Mathematical Society, Providence, RI, 2004. 

\bibitem[CYZ]{CYZ}
S.Y.A. Chang, P. Yang, R. Zhang, 
{Rigidity of Poinca\'{e}-Einstein manifolds with cylindracal conformal infinity},
Preprint, 2025.


\bibitem[CT]{CT} J. Cheeger, G. Tian, 
{On the cone structure at infinity of Ricci flat manifolds
with Euclidean volume growth and quadratic curvature decay}, 
Invent. Math. 118 (1994) 493-571. 


\bibitem[DJ]{DJ}
S. Dutta, M. Javaheri,
{Rigidity of conformally compact manifolds with the round sphere as conformal infinity}. 
Adv Math, 224 (2010), 525-538.


\bibitem[FG]{FG}
C. Fefferman, C. R. Grahamm ,
{$Q$-curvature and Poincar\'{e} metrics},
 Math. Res. Lett., 9(23):139–151, 2002.



\bibitem[Gu]{Gu}
C. Guillarmou,
{Meromorphic properties of the resolvent on asymptotically hyperbolic manifolds},
Duke Math. J.  129 (2005), no. 1, 1-37.

\bibitem[Gr]{Gr} 
C.R. Graham, 
{Volume and area renormalizations for conformally compact Einstein metrics}
, Rend. Circ. Mat. Palermo, Ser.II, Suppl. 63 (2000), 31-42.

\bibitem[GH]{GH}
M. Gursky, Q. Han, 
{Non-existence of Poincar\'{e}-Einstein manifolds with prescribed conformal infinity},
Geom. Funct. Anal. 27 (2017), no.4, 863-879.


\bibitem[GHS]{GHS}
M. Gursky, Q. Han, S. Stolz,
{An invariant related to the existence of conformally compact Einstein fillings},
Trans. Amer. Math. Soc. 374 (2021), no.6, 4185-4205.


\bibitem[GL]{GL}
C.R. Graham, J. Lee, 
{Einstein metrics with prescribed conformal infinity on the ball},
Adv. Math. 87 (1991), no.2, 186-225.

\bibitem[GQ]{GQ}
C. Guillarmou, J. Qing,
{Spectral characterization of Poincar\'{e}-Einstein manifolds with infinity of positive Yamabe type},
Int. Math. Res. Not.  2010, no. 9, 1720-1740.

\bibitem[GS]{GS}
M. Gursky, G. Szekelyhidi, 
{A local existence result for Poincar\'{e}-Einstein metrics}, 
Adv. Math. 361 (2020), 106912, 50 pp.

\bibitem[GZ]{GZ}
C.R. Graham, M. Zworski, 
{Scattering matrix in conformal geometry},
Invent.  Math. 152 (2003), no.1, 89-118.


\bibitem[JS]{JS}
M.S. Josh, A.S. Barreto, 
{Inverse scattering on asymptotically hyperbolic manifolds}.
Acta Math. 184 (2000), no.1, 41-86.

\bibitem[Le1]{Le1}
J. Lee,  
{The spectrum of an asymptotically hyperbolic Einstein manifold},
Comm. Anal. Geom. 3 (1995), no.1-2, 253-271.

\bibitem[Le2]{Le2}
M. Lee,
{Fredholm operators and Einstein metrics on conformally compact manifolds}. 
Mem. Amer. Math. Soc.183(2006), no.864,  83 pp.

\bibitem[L]{L} S. Lee,
{ Liouville-Type Theorems on the Hyperbolic Space},
 arXiv preprint arXiv:2401.06623.


\bibitem[Li1]{Li1} G. Li,
{On uniqueness of conformally compact Einstein metrics with homogeneous conformal infinity}, Adv Math, 2018,
340: 983-1011.

\bibitem[Li2]{Li2} G. Li,
{On uniqueness and existence of conformally compact Einstein metrics with homogeneous conformal infinity}, Calc
Var Partial Differential Equations, 2022, 61: 60.



\bibitem[LQS]{LQS}
 G. Li, J. Qing, Y. Shi, 
{Gap phenomena and curvature estimates for conformally compact Einstein manifolds}. 
Trans. Amer. Math. Soc, 2017,  {369}(6), 4385-4413.

\bibitem[Ma]{Ma}
R. Mazzeo,
{Unique continuation at infinity and embedded eigenvalues for asymptotically hyperbolic manifolds.}
Amer. J. Math.  113 (1991), no. 1, 25-45.

 \bibitem[MM]{MM}
 R. Mazzeo, R. Melrose,
 {Meromorphic extension of the resolvent on complete spaces with asymptotically constant negative curvature}, J. Funct. Anal., 75(1987), 260-310. 

\bibitem[Qi2]{Qi}
J. Qing, Jie,
{Asymptotically hyperbolic manifolds and conformal geometry.
Recent developments in geometry and analysis}, 329-344.
Adv. Lect. Math. (ALM), 23
International Press, Somerville, MA, 2012.

\bibitem[Qi]{Qi1}
J. Qing, 
{On the rigidity for conformally compact Einstein manifolds},
Int. Math. Res. Not. 2003, no. 21, 1141-1153.


\bibitem[ST]{ST}
Y. Shi , G. Tian,
{Rigidity of asymptotically hyperbolic manifolds}. 
Comm. Math. Phys., 2005, {259}, 545-559.

\bibitem[STV1]{STV1} Y. Sire, S. Terracini, S. Vita, {Liouville type theorems and regularity of solutions to degenerate or singular
problems part I: even solutions},
 Comm. Partial Differential Equations 46-2 (2021), 310-361.

\bibitem[STV2]{STV2}  Y. Sire, S. Terracini, S. Vita, 
{Liouville type theorems and regularity of solutions to degenerate or singular
problems part II: odd solutions},
 Math. Eng. 3-1 (2021), 1-50.
 
 
 
 
\bibitem[TV1]{TV1} G. Tian, J. Viaclovsky,
{Bach-flat asymptotically locally Euclidean metrics},
 Invent. Math., 160(2) (2005),357–415.

\bibitem[TV2]{TV2} G. Tian, J. Viaclovsky, 
{Moduli spaces of critical Riemannian metrics in dimension four},
 Adv. Math., 196(2) (2005), 346–372.




\bibitem[WW]{WW}
X. Wang, Z. Wang, 
{On a sharp inequality relating Yamabe invariants on a Poincare-Einstein manifold},
Proc. Amer. Math. Soc.   150 (2022), no. 11, 4923-4929.

 
\bibitem[WZ]{WZ}
F. Wang, H. Zhou, 
{A note on the compactness of Poincar\'{e}-Einstein manifolds},
Commun. Contemp. Math.25(2023), no.5, Paper No. 2250015, 35 pp.











\end{thebibliography}
\end{document}